\documentclass[a4paper,12pt]{amsart}

\usepackage{amsmath,amssymb,amsbsy,amsfonts,amsthm,latexsym,amsopn,amstext,amsxtra,euscript,amscd,cite,color,bm}

\usepackage[colorlinks,linkcolor=blue,anchorcolor=blue,citecolor=blue,backref=page]{hyperref}

\usepackage{color}

\renewcommand*{\backref}[1]{}
\renewcommand*{\backrefalt}[4]{%
    \ifcase #1 (Not cited.)%
    \or        (p.\,#2)%
    \else      (pp.\,#2)%
    \fi}

\hypersetup{breaklinks=true}

\begin{document}

\def\ov#1{{\overline{#1}}}
\def\un#1{{\underline{#1}}}
\def\wh#1{{\widehat{#1}}}
\def\wt#1{{\widetilde{#1}}}

\newcommand{\Ch}{{\operatorname{Ch}}}
\newcommand{\Elim}{{\operatorname{Elim}}}
\newcommand{\proj}{{\operatorname{proj}}}
\newcommand{\h}{{\operatorname{h}}}

\newcommand{\hh}{\mathrm{h}}
\newcommand{\aff}{\mathrm{aff}}
\newcommand{\Spec}{{\operatorname{Spec}}}
\newcommand{\Res}{{\operatorname{Res}}}
\newcommand{\Orb}{{\operatorname{Orb}}}

\newcommand{\hcan}{{\operatorname{\wh h}}}

\newcommand{\hooklongrightarrow}{\lhook\joinrel\longrightarrow}

\newcommand{\bfa}{{\boldsymbol{a}}}
\newcommand{\bfb}{{\boldsymbol{b}}}
\newcommand{\bfc}{{\boldsymbol{c}}}
\newcommand{\bfd}{{\boldsymbol{d}}}
\newcommand{\bff}{{\boldsymbol{f}}}
\newcommand{\bfg}{{\boldsymbol{g}}}
\newcommand{\bfell}{{\boldsymbol{\ell}}}
\newcommand{\bfp}{{\boldsymbol{p}}}
\newcommand{\bfq}{{\boldsymbol{q}}}
\newcommand{\bfs}{{\boldsymbol{s}}}
\newcommand{\bft}{{\boldsymbol{t}}}
\newcommand{\bfu}{{\boldsymbol{u}}}
\newcommand{\bfv}{{\boldsymbol{v}}}
\newcommand{\bfw}{{\boldsymbol{w}}}
\newcommand{\bfx}{{\boldsymbol{x}}}
\newcommand{\bfy}{{\boldsymbol{y}}}
\newcommand{\bfz}{{\boldsymbol{z}}}

\newcommand{\bfA}{{\boldsymbol{A}}}
\newcommand{\bfF}{{\boldsymbol{F}}}
\newcommand{\bfG}{{\boldsymbol{G}}}
\newcommand{\bfQ}{{\boldsymbol{Q}}}
\newcommand{\bfR}{{\boldsymbol{R}}}
\newcommand{\bfU}{{\boldsymbol{U}}}
\newcommand{\bfX}{{\boldsymbol{X}}}
\newcommand{\bfY}{{\boldsymbol{Y}}}
\newcommand{\bfZ}{{\boldsymbol{Z}}}

\newcommand{\bfeta}{{\boldsymbol{\eta}}}
\newcommand{\bfxi}{{\boldsymbol{\xi}}}
\newcommand{\bfrho}{{\boldsymbol{\rho}}}

\def\fE{{\mathfrak E}}



\newfont{\teneufm}{eufm10}
\newfont{\seveneufm}{eufm7}
\newfont{\fiveeufm}{eufm5}
%
%
\newfam\eufmfam
                \textfont\eufmfam=\teneufm \scriptfont\eufmfam=\seveneufm
                \scriptscriptfont\eufmfam=\fiveeufm
%
%
\def\frak#1{{\fam\eufmfam\relax#1}}
%

\def\ts{\thinspace}

\newtheorem{theorem}{Theorem}
\newtheorem{lemma}[theorem]{Lemma}
\newtheorem{claim}[theorem]{Claim}
\newtheorem{cor}[theorem]{Corollary}
\newtheorem{prop}[theorem]{Proposition}
\newtheorem{question}[theorem]{Open Question}

\newtheorem{rem}[theorem]{Remark}
\newtheorem{definition}[theorem]{Definition}

\numberwithin{table}{section}
\numberwithin{equation}{section}
\numberwithin{figure}{section}
\numberwithin{theorem}{section}


\def\squareforqed{\hbox{\rlap{$\sqcap$}$\sqcup$}}
\def\qed{\ifmmode\squareforqed\else{\unskip\nobreak\hfil
\penalty50\hskip1em\null\nobreak\hfil\squareforqed
\parfillskip=0pt\finalhyphendemerits=0\endgraf}\fi}

\def\fA{{\mathfrak A}}
\def\fB{{\mathfrak B}}

\def\cA{{\mathcal A}}
\def\cB{{\mathcal B}}
\def\cC{{\mathcal C}}
\def\cD{{\mathcal D}}
\def\cE{{\mathcal E}}
\def\cF{{\mathcal F}}
\def\cG{{\mathcal G}}
\def\cH{{\mathcal H}}
\def\cI{{\mathcal I}}
\def\cJ{{\mathcal J}}
\def\cK{{\mathcal K}}
\def\cL{{\mathcal L}}
\def\cM{{\mathcal M}}
\def\cN{{\mathcal N}}
\def\cO{{\mathcal O}}
\def\cP{{\mathcal P}}
\def\cQ{{\mathcal Q}}
\def\cR{{\mathcal R}}
\def\cS{{\mathcal S}}
\def\cT{{\mathcal T}}
\def\cU{{\mathcal U}}
\def\cV{{\mathcal V}}
\def\cW{{\mathcal W}}
\def\cX{{\mathcal X}}
\def\cY{{\mathcal Y}}
\def\cZ{{\mathcal Z}}

\def\nrp#1{\left\|#1\right\|_p}
\def\nrq#1{\left\|#1\right\|_m}
\def\nrqk#1{\left\|#1\right\|_{m_k}}
\def\Ln#1{\mbox{\rm {Ln}}\,#1}
\def\nd{\hspace{-1.2mm}}
\def\ord{{\mathrm{ord}}}
\def\Cc{{\mathrm C}}
\def\Pb{\,{\mathbf P}}

\def\va{{\mathbf{a}}}

\newcommand{\commA}[1]{\marginpar{%
\begin{color}{red}
\vskip-\baselineskip 
\raggedright\footnotesize
\itshape\hrule \smallskip A: #1\par\smallskip\hrule\end{color}}}
\newcommand{\commAA}[1]{\marginpar{%
\begin{color}{magenta}
\vskip-\baselineskip 
\raggedright\footnotesize
\itshape\hrule \smallskip A: #1\par\smallskip\hrule\end{color}}}

\newcommand{\commC}[1]{\marginpar{%
\begin{color}{green}
\vskip-\baselineskip 
\raggedright\footnotesize
\itshape\hrule \smallskip I: #1\par\smallskip\hrule\end{color}}}

\newcommand{\commI}[1]{\marginpar{%
\begin{color}{blue}
\vskip-\baselineskip 
\raggedright\footnotesize
\itshape\hrule \smallskip I: #1\par\smallskip\hrule\end{color}}}

\newcommand{\commII}[1]{\marginpar{%
\begin{color}{magenta}
\vskip-\baselineskip 
\raggedright\footnotesize
\itshape\hrule \smallskip II: #1\par\smallskip\hrule\end{color}}}

\newcommand{\commM}[1]{\marginpar{%
\begin{color}{magenta}
\vskip-\baselineskip 
\raggedright\footnotesize
\itshape\hrule \smallskip CM: #1\par\smallskip\hrule\end{color}}}




\newcommand{\ignore}[1]{}

\def\vec#1{\boldsymbol{#1}}

\def\e{\mathbf{e}}



\def\GL{\mathrm{GL}}

\hyphenation{re-pub-lished}

\def\rank{{\mathrm{rk}\,}}
\def\dd{{\mathrm{dyndeg}\,}}
\def\lcm{{\mathrm{lcm}\,}}

\def\A{\mathbb{A}}
\def\B{\mathbf{B}}
\def \C{\mathbb{C}}
\def \F{\mathbb{F}}
\def \K{\mathbb{K}}

\def \Z{\mathbb{Z}}
\def \P{\mathbb{P}}
\def \R{\mathbb{R}}
\def \Q{\mathbb{Q}}
\def \N{\mathbb{N}}
\def \Z{\mathbb{Z}}

\def \L{\mathbf{L}}

\def \nd{{\, | \hspace{-1.5 mm}/\,}}

\def\mand{\qquad\mbox{and}\qquad}

\def\Zn{\Z_n}

\def\Fp{\F_p}
\def\Fq{\F_q}
\def \fp{\Fp^*}
\def\\{\cr}
\def\({\left(}
\def\){\right)}
\def\fl#1{\left\lfloor#1\right\rfloor}
\def\rf#1{\left\lceil#1\right\rceil}
\def\vh{\mathbf{h}}
\def\e{\mbox{\bf{e}}}
\def\ed{\mbox{\bf{e}}_{d}}
\def\ek{\mbox{\bf{e}}_{k}}
\def\eM{\mbox{\bf{e}}_M}
\def\emd{\mbox{\bf{e}}_{m/\delta}}
\def\eqk{\mbox{\bf{e}}_{m_k}}
\def\ep{\mbox{\bf{e}}_p}
\def\eps{\varepsilon}
\def\er{\mbox{\bf{e}}_{r}}
\def\et{\mbox{\bf{e}}_{t}}
\def\Kc{\,{\mathcal K}}
\def\Ic{\,{\mathcal I}}
\def\Bc{\,{\mathcal B}}
\def\Rc{\,{\mathcal R}}

\def\ZK{\Z_\K}
\def\LH{\cL_H}

\def \fI{\mathfrak{I}}
\def \fJ{\mathfrak{J}}
\def \fV{\mathfrak{V}}

\title[Polynomial Equations and Dynamical Systems]
{Reductions Modulo Primes of Systems of Polynomial Equations and
Algebraic Dynamical Systems}

\author[D'Andrea]{Carlos D'Andrea}
\address{Departament  de Matem\`atiques i Inform\`atica, Universitat de
  Barcelona (UB).
Gran Via~585, 08007 Barcelona, Spain}
\email{cdandrea@ub.edu}
\urladdr{\url{http://atlas.mat.ub.es/personals/dandrea/}}

\author[Ostafe] {Alina Ostafe} 
\address{School of Mathematics and Statistics, University of New South Wales. 
Sydney, NSW 2052, Australia}
\email{alina.ostafe@unsw.edu.au}
\urladdr{\url{http://web.maths.unsw.edu.au/~alinaostafe/}}

\author[Shparlinski]{Igor E. Shparlinski} 
\address{School of Mathematics and Statistics, University of New South Wales.
Sydney, NSW 2052, Australia}
\email{igor.shparlinski@unsw.edu.au}
\urladdr{\url{http://web.maths.unsw.edu.au/~igorshparlinski/}}

\author[Sombra]{Mart{\'\i}n~Sombra}
 \address{ICREA. Passeig Llu\'is
  Companys 23, 08010 Barcelona, Spain \vspace*{-2.5mm}}
\address{Departament de Matem\`atiques i Inform\`atica, Universitat de
  Barcelona (UB). Gran Via 585, 08007
  Bar\-ce\-lo\-na, Spain} 
\email{sombra@ub.edu}
\urladdr{\url{http://atlas.mat.ub.es/personals/sombra/}}



\thanks{D'Andrea was partially supported by the Spanish MEC research
  project MTM2013-40775-P, Ostafe by the UNSW Vice Chancellor's
  Fellowship, Shparlinski by the Australian Research Council
  Grants~DP140100118 and~DP170100786, and Sombra by the Spanish MINECO
  research projects MTM2012-38122-C03-02 and MTM2015-65361-P}  

\date{}

\begin{abstract}
We give bounds for the number and the size of the
  primes $p$ such that a reduction modulo $p$ of a system of
  multivariate polynomials over the integers with a finite number $T$ of
  complex zeros, does not have exactly $T$ zeros over the algebraic
  closure of the field with $p$ elements.

  We apply these bounds to study the periodic points and the intersection of
  orbits of algebraic dynamical systems over finite fields.
  In particular, we establish some links between these problems and
  the uniform dynamical Mordell--Lang conjecture. 
\end{abstract}

\maketitle


\section{Introduction}

The goal of the paper is to extend the scope of application of 
algebraic geometric methods to algebraic dynamical systems, that 
is, to dynamical systems generated by iterations of rational functions. 

Let
$$
\bfR = (R_1,\ldots,R_m), \qquad R_1,\ldots,R_m \in \Q(\bfX), 
$$
be a system of 
$m$ rational functions 
in $m$ variables $\bfX=( X_1,\ldots,X_m)$ over $\Q$. 
The iterations of this system of rational functions are given by
\begin{equation}
\label{eq:RatIter}
R_i^{(0)}=X_i \mand
R_i^{(n)}= R_i(R_1^{(n-1)}, \ldots ,R_m^{(n-1)})
\end{equation}
for $i=1, \ldots ,m$ and $n\ge1$, as long as the compositions are
well-defined. We refer to~\cite{AnKhr,Schm,Silv1} for a background on
the dynamical systems associated with these iterations.

For $i=1, \ldots ,m$ and $n\ge1$ write 
\begin{equation}
\label{eq:RedForm}
R_i^{(n)}= \frac{F_{i,n}}{G_{i,n}}
\end{equation}
with coprime
$F_{i,n}, G_{i,n}\in\Z[\bfX]$ and $G_{i,n}\ne 0$.
Given a prime $p$ such that  $G_{i,j}\not\equiv0 \pmod{p}$,
$j=1,\ldots, n$, we can
consider the reduction modulo $p$ of the iteration~\eqref{eq:RatIter}. 
Recently, there have been many advances in the study of periodic
points and period lengths in reductions of orbits of dynamical systems
modulo distinct primes $p$~\cite{AkGh,BGHKST,Jon,RoVi,Silv2}. However,
many important questions remain widely open, including:
\begin{itemize}
\item the distribution of the period length, 
\item the number of periodic points, 
\item the number of common values in orbits of two distinct 
algebraic dynamical systems.
\end{itemize}

Furthermore, some of our motivation comes from the recently introduced 
idea of transferring the {\it Hasse principle\/} for periodic points and
thus linking local and global periodicity properties~\cite{Tow}.

In this paper, we use several tools from arithmetic geometry to obtain
new results about the orbits of the reductions modulo a prime $p$ of
algebraic dynamical systems.  Our approach is based on a new result
about the reduction modulo prime numbers of systems of multivariate
polynomials over the integers. If the system has a finite number of
solutions $T$ over the complex numbers, then there exists a positive
integer $\fA$ such that, for all prime numbers $p\nmid\fA$, the
reduction modulo $p$ of the system has also $T$ solutions over
$\ov \F_{p}$, the algebraic closure of the field with $p$ elements.t
Here, using an arithmetic version of {\it Hilbert's
  Nullstellensatz\/}~\cite{DKS,KPS} and elimination theory, we give a
bound, in terms of the degree and the height of the input polynomials,
for the integer $\fA$ controlling the primes of bad reduction
(Theorem~\ref{thm:1}).   For $T=0$, that is, for systems of polynomial
equations without solutions over $\C$, this question has been
previously adressed in~\cite{HMPS,Koi}. Indeed, the corresponding
bound for the set of primes of bad reduction was a key step in
Koiran's proof that, from the point of view of complexity theory, the
satisfability problem for systems of polynomial equations lies in the
polynomial hierarchy~\cite{Koi}.

As an immediate application of this result, in
Theorems~\ref{thm:Cycles} and~\ref{thm:Cycles-Poly} we bound the
number of points of given period in the reduction modulo $p$ of an
algebraic dynamical system.
Also, combining Theorem~\ref{thm:1} with some combinatorial arguments,
we give in Theorem \ref{thm:OrbIntVar} a bound for the frequency of the points in an orbit of the
reduction modulo $p$ of an algebraic dynamical system lying in a given
algebraic variety, or that coincide with a similar point coming from a
trajectory of another algebraic dynamical system. 

We also use a different
approach, based again on an explicit version of Hilbert's
Nullstellensatz, to obtain  in Theorem~\ref{thm:OrbIntVar Alt} better results for the problem of bounding the
frequency of the points in an orbit lying in a given algebraic
variety,  under a different and apparently more
restrictive condition. 

Our  bounds 
are uniform in the prime $p$, provided that $p$ avoids a certain set
of exceptions. In particular, our bounds for the number
of $k$-periodic points can be viewed as distant relatives of the {\it
  Northcott theorem\/} for dynamical systems
in~\cite[Theorem~3.12]{Silv1}, which bounds the number of pre-periodic
points in algebraic dynamical systems over finite algebraic extensions
of $\Q$.   Here we restrict the
length of the period, but instead we consider all $k$-periodic points
over~$\ov \F_p$.

From a computational point of view, the arithmetic Nullstellens\"atze
in~\cite{DKS,KPS} are effective. Using this, one can show that the
positive integers describing the set of exceptional primes in our
results can be effectively computed.

Further applications of our results have been 
given in our subsequent paper with Chang~\cite{CDOSS}.


\subsection*{Acknowledgements.}  We are grateful to Dragos Ghioca,
Luis Miguel Pardo, Richard Pink, Thomas Tucker and Michael Zieve for
many valuable discussions and comments, specially concerning the
plausibility of the uniform boundedness assumption for the orbit
intersections.

\section{Modular Reduction of Systems of Polynomial Equations}
\label{sec:modul-reduct-syst}

\subsection{General notation}

Throughout this text, boldface letters denote finite sets or
sequences of objects, where the type and number should be clear from
the context. In particular, $\bfX$ denotes the group of variables
$(X_{1},\ldots, X_{m})$, so that $\Z[\bfX]$ denotes the ring of
polynomials $\Z[X_{1},\ldots,X_{m}]$ and $\Q(\bfX)$ the field of
rational functions $\Q(X_1,\ldots,X_m)$.

We denote by $\N$ the set of positive integer numbers. Given
functions 
$$
f,g\colon \N\longrightarrow \R, 
$$
the symbols $f=O(g)$ and $f\ll g$ both
mean that there is a constant $c\ge0$ such that $|f(k)| \le c \, g(k)$
for all $k\in \N$.  To emphasize the dependence of the implied
constant $c$ on parameters, say $m$ and $s$, we write $f=O_{m,s}(g)$
or $f\ll_{m,s} g$. We use the same convention for other parameters as
well.

For a polynomial $F\in \Z[\bfX]$, we define its
\emph{height}, denoted by $\h(F)$, as the logarithm of the maximum of
the absolute values of its coefficients.  For a rational function
$R\in \Q(\bfX)$, we write $R=F/G$ with coprime
$F,G\in\Z[\bfX]$ and we define the \emph{degree} and the
\emph{height} of $R$ respectively as the maximum of the degrees and of
the height of $F$ and $G$, that is,
$$
\deg R = \max\{\deg F, \deg G\}
\mand
\h(R)= \max\{\h(F), \h(G)\}. 
$$

Let  $K$ be a field and  $\ov K$ its algebraic closure. Given a
family of polynomials $G_1,\ldots,G_s \in K[\bfX]$, we
denote by
$$
  V(G_{1},\ldots, G_{s})=\Spec\( K[\bfX]/(G_{1},\ldots,
  G_{s})\) \subset \A^{m}_{K}
$$
its associated affine algebraic variety. We also denote by
$Z(G_1,\ldots,G_s)$ their zero set in $\ov K^{m}$, which coincides
with the set of $\ov K$-valued points $V(G_{1},\ldots, G_{s})(\ov K)$.

Let
\begin{equation} \label{eq:16}
\bfR = (R_1,\ldots,R_m), \qquad R_1,\ldots,R_m \in K(\bfX)
\end{equation} 
a system of $m$ rational functions in $m$ variables over $K$. For
$n\ge 1$, we denote
$$
\bfR^{(n)} = (R_1^{(n)}, \ldots ,R_m^{(n)}),
$$
as long as this iteration is well-defined as a rational function.



Given a point $\bfw \in \ov K^m$ we define its orbit with 
respect to the system of rational functions above as the set
\begin{multline}
\label{eq:Orb R}
\Orb_\bfR(\bfw) = \{\bfw_{n} \mid \text{with} \ \bfw_{0}= \bfw \
 \\ \text{and}
\
\bfw_{n} = \bfR\(\bfw_{n-1}\), \ n =1,2,\ldots \}. 
\end{multline}
The orbit terminates if $\bfw_{n}$ is a pole of $\bfR$
and, in this case, $\Orb_\bfR(\bfw)$ is a finite set. 

If the point $\bfw_{n}$ in~\eqref{eq:Orb R} is defined, then
$\bfw_{0}$ is not a pole of $\bfR^{(n)}$ and
$\bfw_{n}=\bfR^{(n)}(\bfw_{0})$. However, the fact that the evaluation
$\bfR^{(n)}(\bfw_{0})$ is defined does not imply the existence of
$\bfw_{n}$, since this latter point is defined if and only if all the
previous points of the orbit~\eqref{eq:Orb R} are defined and
$\bfw_{n-1}$ is not a pole of $\bfR$.  For instance, let $m=1$ and
$R(X)=1/X$. Then $R^{(2)}(X)=X$ and we see that $R^{(2)}(0)=0$, but
$w_2=R(R(0))$ is not defined as $0$ is a pole for~$R$.  Clearly, for
polynomial systems this distinction does not exist.

 \subsection{Preserving the number of points}

The following is our main result concerning the reduction modulo prime numbers of
systems of multivariate polynomials over the integers.

\begin{theorem} \label{thm:1} Let $m\ge 1$ and  let $F_1,\ldots,F_s\in
  \Z[\bfX]$  be a system of polynomials whose zero set in
  $\C^{m}$ has a finite number $T$ of distinct points.  Set
  $$
  d=\max_{i=1,\ldots,m}\deg F_{i}\mand h=\max_{i=1,\ldots,m}\h(F_{i}).
  $$  
  Then there
  exists $\fA \in \N$ satisfying 
$$
\log \fA \le C_1(m) d^{3m +1} h +C_2(m,s)d^{3m +2} ,
$$
with
$$
C_1(m) =11m  + 4 \mand
C_2(m,s) =  (55 m +99) \log ((2m+5)s)
$$
 and such that, if $p$ is a prime number not dividing $\fA$, then the
 zero set in $\ov \F_{p}^{m}$ of the system of polynomials
 $F_i\pmod{p}$, $i=1,\ldots, s$, consists of exactly $ T$ distinct
 points.
\end{theorem}

This result allows us to control the number and the height of the
primes of bad reduction.

\begin{cor} \label{cor:2} With notation as in Theorem~\ref{thm:1}, set
  $\bfF=(F_{1},\ldots,F_{s})$ and let $S_{\bfF}$ denote the set of
  prime numbers such that the number of zeros in $\ov \F_{p}^{m}$
  of the system of polynomials $F_i\pmod{p}$, $i=1,\ldots, s$, is
  different from $T$. Then
$$
 \max \left\{ \# S_{\bfF}, \, \max_{p \in S_{\bfF}}\log p\right\} \ll_{m,s}d^{3m+1} h + d^{3m+2}.
$$
\end{cor}

\begin{rem} \label{rem:1}
In the interesting special case when $s=m$, one can get a
slightly stronger version of Theorem~\ref{thm:1}, but of the same
general shape. 

It is also very plausible that Theorem~\ref{thm:1} admits a number of
extensions such as zero-dimensional systems of polynomial equations on
an equidimensional variety $X\subseteq \A_{\C}^{m}$ instead of just on
$\A_{\C}^{m}$.  One can also obtain a bound taking into account the
degree and the height of each individual polynomials $F_{j}$.
\end{rem}

\subsection{Preliminaries} 
\label{sec:preliminaries}

Besides the application of an arithmetic Nullstellensatz, the proof of
Theorem~\ref{thm:1} relies on elimination theory and on the basic
properties of schemes over the integers. Hence, it is convenient
to work using the language of algebraic geometry as in, for
instance,~\cite{Liu02}.

Let $F_1,\ldots,F_s\in \Z[\bfX]$ be a system of polynomials
whose zero set in $\C^{m}$ has a finite number $T$ of distinct points,
as in the statement of Theorem~\ref{thm:1}. Denote by $V$ the
subvariety of the affine space $\A^{m}_{\Q}=\Spec(\Q[\bfX])$ defined by this system of
polynomials. For each prime $p$, set
\begin{equation}
\label{eq:Fip}
 F_{i,p}\in \F_{p}[\bfX]  
\end{equation}
for the reduction modulo $p$ of $F_{i}$, and by $V_{p}$ the subvariety
of $\A^{m}_{\F_{p}}=\Spec(\F_{p }[\bfX])$ defined by the
system $F_{i,p}$, $i=1,\ldots, s$.

Recall that, given a field extension $K\hookrightarrow L$ and a
variety $X$ over $K$, we denote by $X(L)$ the set of $L$-valued points
of $X$. We have that
$$
  \A_{\Q}^{m}(\C)= \C^{m} \quad \text{ and } \quad \A_{\F_{p}}^{m}(\ov
  \F_{p})= \ov \F_{p}^{m},
$$
and that the varieties $ V(\C)$ and $ V_{p}(\ov\F_{p})$ coincide with the zero sets
$ Z(F_{1},\ldots, F_{s})$ and $Z(F_{1,p},\ldots, F_{s,p})$,
respectively. Our aim is to give a bound for an integer $\fA\in \N$
such that, if $p\nmid \fA$, then $V_{p}(\ov\F_{p})$ consists of $T$
distinct points.

Let $\A^{m}_{\Z}$ and $\P^{m}_{\Z}$ be the affine space and the
projective space over the integers, respectively.  We denote by
$\bfZ=\{Z_{0},\ldots, Z_{m}\}$ the homogeneous coordinates of
$\P^{m}_{\Z}$. Using the standard inclusion
\begin{equation}\label{eq:8}
\iota\colon \A^{m}_{\Z}\hooklongrightarrow \P^{m}_{\Z}\quad , \quad (x_{1},\ldots,x_{m})\longmapsto (1:x_{1}:\ldots:x_{m}),
\end{equation}
we identify $\A_{\Z}^{m}$ with the open subset of $\P^{m}_{\Z}$ given
by the non-vanishing of $Z_{0}$. The coordinates of these spaces are
then related by $X_{i}=Z_{i}/Z_{0}$.

Let $\cV$ and $\ov \cV$ denote the closure of $V$ in $\A^{m}_{\Z}$ and
in $\P^{m}_{\Z}$, respectively. Then $\cV$ is the affine scheme
corresponding to the ideal
$$
I(\cV)=I(V) \cap \Z[\bfX]
$$
and $\ov \cV$ is the projective scheme corresponding to 
$$
I(\ov \cV) =I(\cV)^{\hh}\subseteq \Z[\bfZ],  
$$
the homogenisation of the ideal $I(\cV)$.

Consider the projection $\pi\colon \P^{m}_{\Z}\to \Spec(\Z)$ and set
$$
 \ov \cV_{p}=\pi^{-1}(p) \cap \ov
\cV 
$$
for the fibre over the prime $p$ of the restriction to $\ov \cV$ of this map.
It is a subscheme of the projective space $ \P^{m}_{\F_{p}}$.

The morphism of schemes $\cV\to \Spec(\Z)$ is \emph{flat} if there is
a family of flat $\Z$-algebras $\cA_{j}$, $j\in J$, such that their
associated affine schemes form an open covering of $\cV$, namely $
\cV=\bigcup_{j\in J}\Spec(\cA_{j})$. Since $\Z$ is a principal ideal
domain, the $\Z$-algebras $\cA_{j}$, $j\in J$, are flat if and only if
they are torsion free~\cite[Corollary~1.2.5]{Liu02}. The flatness of
an algebra over a ring is a property concerning extensions of
scalars. At the geometric level, this property ensures a certain
continuity behavior of the fibres of the morphism,
see~\cite[\S4.3]{Liu02} for more details.

\begin{lemma} \label{lem:3} Let notation be as above. 
  \begin{enumerate}
  \item \label{item:4} The projective scheme $\ov \cV$ is flat over
    $\Spec(\Z)$ and moreover, it is reduced, has pure relative dimension
    0, and none of its irreducible components is contained in the
    hyperplane at infinity.
\item \label{item:5} For all $p\in\Spec(\Z)$, we have that 
$\ov \cV_{p}$ is a 0-dimensional subscheme of $ \P^{m}_{\F_{p}}$ of
  degree $T$.
\item \label{item:6} The inclusion $\ov \cV_{p}(\ov \F_{p}) \cap
  {\ov \F_{p}^{m}} \subseteq V_{p}(\ov \F_{p})$ holds.
  \end{enumerate}
\end{lemma}

\begin{proof}
  For the statement~\eqref{item:4}, consider the decomposition $V=\bigcup_{C}C$ into
  irreducible components. For each $C$, denote by its closure in
  $\P^{m}_{\Z}$. Then
$$
  I(\ov \cC)= (I(C)\cap \Z[\bfX])^{\hh}
  \subseteq \Z[\bfZ],
$$
where, as before, $J^\hh$ denotes the homogenisation of the ideal $J$.

One can verify that this ideal is prime and that $I(\ov \cC)\cap
\Z=\{0\}$. We have that 
$$
\ov\cV= \bigcup_{C}\ov \cC,
$$ 
and so $\ov \cV$ is
a reduced scheme that, by~\cite[Proposition~4.3.9]{Liu02}, is flat
over $\Spec(\Z)$.  Moreover, the Krull dimension of the quotient ring
$\Z[\bfZ]/I(\ov\cC)$ is one and $Z_{0}\notin I(\ov\cC)$,
which respectively implies that $\ov\cV$ is of pure relative dimension
0 and that none
of its irreducible components is contained in the hyperplane at
infinity of $\P^{m}_{\Z}$, as stated. 

Now we turn to the statement~\eqref{item:5}. By the invariance of the
Euler-Poincar\'e characteristic of the fibres of a projective flat
morphism, see~\cite[Proposition~5.3.28]{Liu02}, and the fact that the map
$\ov \cV\to \Spec(\Z)$ is flat, the Hilbert polynomial of $\ov
\cV_{p}$ coincides with that of the generic fibre of that map. This
generic fibre coincides with the closure of $V$ in $\P^{m}_{\Q}$,
which is a 0-dimensional variety of degree $T$. It follows that its
Hilbert polynomial is the constant $T$, and so $\ov \cV_{p}$ is also a
0-dimensional scheme of degree~$T$.

To prove the statement~\eqref{item:6}, note first that $ \ov
\cV_{p}(\ov\F_{p})$ is given by the zero set in $\P^{m}_{\F_{p}}(\ov
\F_{p})$ of the ideal
$$
  \(\sqrt{(F_{1}, \ldots, F_{s})} \cap \Z[\bfX]\)^{\hh} \pmod{p} \quad \subseteq \F_{p}[\bfZ].
$$
Hence, the intersection $\ov \cV_{p}(\ov \F_{p}) \cap {\ov \F_{p}^{m}} $ coincides with
the zero set in $\ov \F_{p}^{m}$ of the affinisation of this ideal,
obtained by setting $Z_{0}\to 1$ and $Z_{i}\to X_{i}$, $i=1,\ldots,
m$.  Denote by $I_{1}$ this ideal of $\F_{p}[\bfX]$.

On the other hand, $V_{p}$
is given by the zero set in $\ov \F_{p}^{m}$ of the ideal
$$
 I_{2}=\sqrt{\smash{(F_{1,p}, \ldots, F_{s,p})}\vphantom{F_1}} \subseteq\F_{p}[\bfX]
$$
with $F_{1,p}, \ldots, F_{s,p}$ as in~\eqref{eq:Fip}.  Then
\eqref{item:6} follows from the inclusion of ideals $I_{1}\supset
I_{2}$.
\end{proof}

\subsection{Eliminants  and heights} \label{sec:eliminants-heights}

We recall the notion of eliminant of a homogeneous ideal as presented
by Philippon in~\cite{Phi86}.  Let $R$ be a principal ideal domain,
with group of units $R^{\times}$ and field of fractions~$K$. Let
$\bfU=\{U_{0},\ldots, U_{m}\}$ be a further group of $m+1$ variables
and consider the general linear form in the variables $\bfZ$ given by
$$
  L=U_{0}Z_{0}+\ldots+U_{m}Z_{m} \in \Z[\bfU][\bfZ].
$$

\begin{definition}
  \label{def:1}
Let $I\subseteq R[\bfZ]$ be a homogeneous ideal. The \emph{eliminant ideal} of
$I$ is the ideal of $R[\bfU]$ defined as 
\begin{multline*}
  \fE(I) = \{ F\in R[\bfU] \mid  \exists k\ge 0  \\\text{ with } Z_{j}^{k} F\in
  I R[\bfU,\bfZ] + (L) \text{ for } j=0,\ldots, m\}.
\end{multline*}
If $\fE(I)$ is principal, then the \emph{eliminant} of $I$, denoted by
$\Elim(I)$, is defined as any generator of this ideal. 
\end{definition}

The eliminant of an ideal of $R[\bfZ]$ is 
a homogeneous polynomial, uniquely defined up to a factor in
$R^{\times}$.

In the following proposition, we gather the basic properties of
eliminants of 0-dimensional ideals following~\cite{Nes77,Phi86}.
Given a prime ideal $P$ in some ring and a $P$-primary ideal $Q$, the
\emph{exponent} of $Q$, denoted by $e(Q)$, is the least integer $e\ge
1$ such that $P^{e}\subseteq Q$.  Notice that $Q$ is prime if and only
if $e(Q)=1$.

\begin{lemma}
  \label{prop:1}
Let $I\subseteq R[\bfZ]$ be an equidimensional homogeneous ideal
defining  a 0-dimensional subvariety of $\P_{K}^{m}$. 
\begin{enumerate}
\item \label{item:1} The eliminant ideal $\fE(I)$ is principal and
  $\Elim(I)$ is well-defi\-ned. 
\item \label{item:7} If $I$ is prime and $(Z_{0},\ldots, Z_{m})
  \not\subset I$, then $\Elim(I)$ is an irreducible polynomial. 
\item \label{item:2} Let $I=\bigcap_{i}Q_{i}$ be the minimal primary
  decomposition of $I$ and set $P_{i}=\sqrt{Q_{i}}$. Then there exists
  $\mu\in R^{\times}$ such that
$$
    \Elim(I)= \mu \prod_{i}\Elim(P_{i})^{e(Q_{i})}.
$$
\item \label{item:3} Let $V(I)(\ov K)$ be the zero set of $I$ in
  $\P^{m}_{K}(\ov K)$. Then
$$
    \Elim(I)= \lambda \prod_{\eta\in V(I)(\ov K)} L(\eta)^{e_{\eta}},
$$
with $\lambda\in K^{\times}$ and where $e_{\eta}$ denotes the exponent
of the primary component associated to the point $\eta$.  In
particular, $I\otimes_{R} K$ is radical if and only if $\Elim(I)$ is
squarefree.
\end{enumerate}
\end{lemma}

\begin{proof} These statements are either contained or can be
  immediately extracted from results in~\cite{Nes77,Phi86}.
  Precisely, the statement~\eqref{item:1}
  is~\cite[Proposition~2(1)]{Nes77} or~\cite[Lemma~1.8]{Phi86}. The
  statement~\eqref{item:7} is contained
  in~\cite[Proposition~1.3(ii)]{Phi86}. The statement~\eqref{item:2} follows
  from~\cite[Corollary to Proposition~3]{Nes77}. The last
  claim~\eqref{item:3} follows from~\eqref{item:2} and the proof
  of~\cite[Lemma~1.8]{Phi86}.
\end{proof}

\begin{lemma}
  \label{lem:Factor}
  Let notation be as in~\S\ref{sec:preliminaries}.  In particular,
  $V$ is the $0$-dimensional subvariety of $\A^{m}_{\Q}$ defined by
  the system $F_{i}$, $i=1,\ldots, s$, $\ov \cV$ its closure in
  $\P^{m}_{\Z}$, and $T$ the number of points in $V(\ov \Q)$.  Then
  $\fE(I(\ov \cV)) $ is a principal ideal and the eliminant
  $\Elim(I(\ov \cV)) \in \Z[\bfU]$ is well-defined. Moreover, this
  eliminant is a primitive polynomial and we have the factorisation
\begin{equation}
  \label{eq:2}
    \Elim(I(\ov \cV)) = \lambda \prod_{
      (\xi_{1},\ldots,\xi_{m}) \in V(\ov \Q)} (U_{0}+\xi_{1}U_{1}+\ldots+ \xi_{m}U_{m})
\end{equation}
with $\lambda\in \Q^{\times}$.
\end{lemma}

\begin{proof}
Set $I=I(\ov \cV) $ for short. The subvariety of $\P^{m}_{\Q}$ defined
by this ideal coincides with $\iota(V)$, the image of $V$ under the standard inclusion
\eqref{eq:8}. This  subvariety is
of dimension 0, and it follows from Lemma~\ref{prop:1}\eqref{item:1} that
the eliminant ideal of $I$ is principal and that its eliminant
polynomial is well-defined. 

By Lemma~\ref{lem:3}\eqref{item:4}, the subscheme $\ov \cV \subset
\P^{m}_{\Z}$ is flat and reduced. Hence, $I=\bigcap_{i}P_{i}$ where
each $P_{i}$ is a prime ideal of $ \Z[\bfZ]$ which defines a
0-dimensional subvariety of $\P^{m}_{\Q}$ and $P_{i}\cap \Z=\{0\}$.
By Lemma~\ref{prop:1}(\ref{item:7},\ref{item:3}) applied to $P_{i}$,
each eliminant $\Elim(P_{i})$ is an nonconstant irreducible
polynomial.  Together with Lemma~\ref{prop:1}\eqref{item:2}, this
implies that $\Elim(I)$ is primitive.

The ideal $I$ is radical and no point of $V$ lies in the hyperplane at
infinity of $\P_{\Q}^{m}$.  Then the factorisation~\eqref{eq:2} follows immediately
from Lemma~\ref{prop:1}(\ref{item:7}, \ref{item:3}).
\end{proof}

 Set 
\begin{equation}
  \label{eq:EV}
E_{V}= \Elim(I(\ov \cV))
\end{equation}
for short.
Our next aim is to bound the height of this polynomial in terms
of the degree and the height of the  $F_{i}$'s. To this end, we first
recall the notion of Weil height of a finite subset of ${\ov\Q}^{m}$.

Given a number field $\K$, we denote by $M_{\K}$ its set of
places. For each $w\in M_{\K}$, we assume the corresponding
absolute value of $\K$, denoted by $|\cdot|_{w}$, extends either the
Archimedean or a $p$-adic absolute value of $\Q$, with their standard
normalisation. 

Let $\bfeta\in \P^{m}_{\Q}(\ov \Q)$ and choose a number field $\K$ such that
$\bfeta=(\eta_{0}:\ldots:\eta_{m})$ with $\eta_{i}\in \K$. 
The \emph{Weil
  height} of $\bfeta$ is defined as
$$
  \hcan (\bfeta)= \sum_{w\in M_{\K}}\frac{[\K_{w}:\Q_{w}]}{[\K:\Q]}
  \log\max\{|\eta_{0}|_{w},\ldots, |\eta_{m}|_{w}\},
$$
where $\K_{w}$ and $\Q_{w}$ denote the $w$-adic completion of $\K$ and
$\Q$, respectively. This formula does not depend neither
on the choice of homogeneous coordinates of $\bfeta$ nor on the number
field $\K$. Hence, it defines a function
$$
  \wh \h\colon \P^{m}_{\Q}(\ov \Q)\longrightarrow \R_{\ge0}.
$$
For a point of $\ov \Q^{m}$, we define its {Weil height} as the Weil
height of its image in $\P_{\Q}^{m}(\ov \Q)$ \emph{via} the
inclusion~\eqref{eq:8} and, for a finite subset of $\ov \Q^{m}$, we
define its {Weil height} as the sum of the Weil height of its points.

Since the $F_{i}$'s have integer
coefficients, the points of $V$ lie in $\ov\Q^{m}$. If we write 
$$
V(\C)=Z(F_{1},\ldots, F_{s}) =\{\bfxi_{1},\ldots, \bfxi_{T}\} 
$$
with $\bfxi_{j}\in \ov\Q^{m}  $, 
then the Weil height of this set is given by 
\begin{equation} \label{eq:1}
  \begin{split} 
  \hcan (V)
&=\sum_{i=1}^{T}\hcan(\bfxi_{i})\\ 
&= \sum_{j=1}^{T}\sum_{w\in M_{\K}}\frac{[\K_{w}:\Q_{w}]}{[\K:\Q]} \log\max\{1,|\xi_{j,1}|_{w},
  \ldots, |\xi_{j,m}|_{w}\}.    
  \end{split}
\end{equation}
We refer to~\cite{BoGu} for a more detailed background on heights.

The notion of Weil height of points extends to
 projective varieties. This extension is  usually called the
  ``normalised'' or ``canonical'' height and also denoted by the
  operator ${\hcan}$, see for instance~\cite[\S{}I.2]{PS08} 
  or~\cite[\S2.3]{DKS}. For an affine variety $Z\subset\A^{m}_{\Q}$,
  we respectively denote by $\deg Z$ and $\hcan(Z)$ the sum of the
  degrees and of the canonical heights of the
  Zariski closure in $\P^{m}_{\ov\Q}$ of its irreducible
  components. We also define the dimension of $Z$, denoted by $\dim
  Z$, as the maximum  of the dimensions of its irreducible
  components. 

The following is a version of the arithmetic B\'ezout inequality. 

\begin{lemma}
  \label{lem:ArithBezout}
  Let $Z\subset\A^{m}_{\Q}$ be a variety and $G_{i}\in \Z[\bfX]$, $i=1,\ldots, t$.  Set 
  $$d_{i}=\deg G_{i}, \qquad 
   h=\max_{i=1,\ldots, t}\h(G_{i}), \qquad m_{0}=\min\{\dim Z,m\},
   $$ and assume that
  $d_{1}\ge\ldots\ge d_{t}$. Then
\begin{multline*}
  \hcan(Z\cap V(G_{1},\ldots, G_{t})) \le \prod_{i=1}^{m_{0}} d_{i}\biggl(\hcan(Z) +
  \biggl(\sum_{i=1}^{m_{0}} \frac{1}{d_{i}} \biggr)  h \deg Z  \\+
  m_{0}\log(m+1) \deg Z\biggr).
\end{multline*}
\end{lemma}

\begin{proof}
  Let $C\subseteq \A^{m}_{\Q}$ be an irreducible subvariety and $
  F\in\Z[\bfX] $ a polynomial such that the
  hypersurface $V(F)\subseteq \A^{m}_{\Q}$ intersects $C$ properly.
  From~\cite[Theorem~2.58]{DKS}, we deduce that
\begin{equation}
  \label{eq:11}
  \hcan(C\cap V(F))\le \hcan(C) \deg F 
 + (\h(F)+
  \deg F  \log(m+1) ) \deg C.
\end{equation}
The stated bound now follows by repeating the scheme of the proof
of~\cite[Corollary~2.11]{KPS} for the canonical height instead of the
Fubini-Study one, and using~\eqref{eq:11} instead of the inequality in
the second line of~\cite[Page~555]{KPS}. 
\end{proof}

Let  $F_{1},\ldots, F_{s}\in \Z[\bfX]$  and let
$V\subseteq\A^{m}_{\Q}$ be the 0-dimensional subvariety defined by this
system of polynomials, as in~\S\ref{sec:preliminaries}. 
Also set 
$$
d=\max_{i=1,\ldots,s}\deg F_{i} \mand h=\max_{i=1,\ldots,s}\h(F_{i}).
$$

\begin{cor}
\label{cor:3}
Write $ V(\C) =\{\bfxi_{1},\ldots, \bfxi_{T}\} $
with $\bfxi_{j}\in \ov\Q^{m}$.  Then
$$
T \le d^{m}\quad \text{ and } \quad   
\sum_{i=1}^{T}\hcan (\bfxi_{i}) \le m d^{m-1}(h+d\log(m+1)).
$$
\end{cor}

\begin{proof}
  The first inequality is given by the B\'ezout theorem.
  For the rest, we have that $\deg \A^{m}_{\Q}=1$ and,
  by~\cite[Proposition~2.39(4)]{DKS}, $\hcan(\A^{m}_{\Q})=0$. The
  statement then follows from Lemma~\ref{lem:ArithBezout} and the
  inequalities $0\le m$ and $d_{i}\le d$.
\end{proof}

\begin{lemma} 
  \label{prop:3}
With notation as above, let $E_{V}$ denote the eliminant of the ideal $I(\ov \cV)$ as
in~\eqref{eq:EV}. Then
$$
 \deg_{U_{0}} E_V
  =\deg E_V = T \le d^{m}     
$$
and
$$
\h(E_V) \le m d^{m-1}h+(m+1)d^{m}\log(m+1).    
$$
\end{lemma}

\begin{proof} 
Set
$$
Q = \prod_{j=1}^{T} (U_{0}+\xi_{j,1}U_{1}+\ldots+\xi_{j,m}U_{m}) \in
\Q[\bfU]
$$
so that, by the factorisation~\eqref{eq:2} of Lemma~\ref{lem:Factor} we have $ E_{V}= \lambda Q $ with
$\lambda\in \Q^{\times}$.  The formula for the degrees of the
eliminant follows readily from this.

For a polynomial $F$ over $ \Q$, we denote
by $\|F\|_{\infty,1}$ the $\ell^{1}$-norm of its vector of coefficients
with respect to the Archimedean absolute value of $\Q$. Then
\begin{equation}
  \label{eq:3}
\h( E_{V})\le \log\|E_{V}\|_{\infty,1} = \log \|Q\|_{\infty,1}+\log |\lambda|_{\infty}.  
\end{equation}
Since $E_{V}$ is primitive, for 
$v\in M_{\Q}\setminus \{\infty\}$,
$$
  0= \log\|E_{V}\|_{v}= \log\|Q\|_{v}+\log|\lambda|_{v},
$$
where $\|Q\|_{v}$ is defined as the maximum norm of  the  vector of the 
coefficients of $Q$ with respect to the absolute value $|\cdot|_{v}$. Summing up over all places and using the product formula, we obtain
\begin{equation} \label{eq:5}
\log\|E_{V}\|_{\infty,1} =  
\log\|Q\|_{\infty,1} +\sum_{v\in M_{\Q}\setminus \{ \infty\} } \log\|Q\|_{v}.
\end{equation}

Let $\K$ be a number field of definition of $\bfxi_{1},\ldots,
\bfxi_{T}$, and denote by $M_{\K}^{\infty}$ and $M_{\K}^{0}$ the set
of Archimedean and non-Archimedean places of~$\K$, respectively.  For
each $w\in M_{\K}^{\infty}$ and a polynomial $F$ over $ \K$, then we
denote by $\|F\|_{w,1}$ the $\ell^{1}$-norm of its vector of
coefficients with respect to the absolute value $|\cdot|_{w}$.  Then,
by the compatibility between places and finite extensions,
  \begin{equation}
 \begin{split}
 \label{eq:7}
\log&\|Q\|_{\infty,1} +\sum_{v\in M_{\Q}\setminus \{ \infty\}  } \log\|Q\|_{v}  \\
& =  
\sum_{w\in M_{\K}^{\infty}}\frac{[\K_{w}:\Q_{w}]}{[\K:\Q]}
\log\|Q\|_{w,1}   + \sum_{w\in M_{\K}^{0}}\frac{[\K_{w}:\Q_{w}]}{[\K:\Q]}   \log\|Q\|_{w}.
     \end{split}
\end{equation}
For  $w\in M_{K}^{\infty}$, by the sub-additivity of
$\log\|\cdot\|_{w,1}$,
  \begin{equation}
 \begin{split}
 \label{eq:9}
  \log&\|Q\|_{w,1}  \le \sum_{j=1}^{T} \log
  \|U_{0}+\xi_{j,1}U_{1}+\ldots+\xi_{j,m}U_{m}\|_{w,1}  \\
&\quad  \le  \sum_{j=1}^{T} \log \max\{1,|\xi_{j,1}|_{w},
  \ldots, |\xi_{j,m}|_{w}\} +T\log(m+1).
     \end{split}
\end{equation}
On the other hand, for $w\in M_{\K}^{0}$, 
  \begin{equation}
 \begin{split}
 \label{eq:10}
    \log\|Q\|_{w} = \sum_{j=1}^{T} \log &
  \|U_{0}+\xi_{j,1}U_{1}+\ldots+\xi_{j,m}U_{m}\|_{w} \\
  & =   \sum_{j=1}^{T} \log \max\{1,|\xi_{j,1}|_{w},
  \ldots, |\xi_{j,m}|_{w}\}.
     \end{split}
\end{equation}
If follows from~\eqref{eq:3}, \eqref{eq:5}, \eqref{eq:7},
\eqref{eq:9}, \eqref{eq:10} and~\eqref{eq:1} that
$$
\h(E_V)\leq \sum_{i=1}^{T}\hcan(\bfxi_{i})+T\log(m+1).
$$
The statement then follows from the bound for the Weil height in
Corollary~\ref{cor:3}.
\end{proof}

Set
$$
  L^{\aff}=U_{0}+U_{1}X_{1}+\ldots+U_{n}X_{n}\in \Z[\bfU,\bfX].
$$
By construction, $E_{V}$
vanishes on the zero locus of $F_{1},\ldots, F_{s}$ and $ L^{\aff}$ in
$\C^{m+1}\times \C^{m}$. 
By Hilbert's Nullstellensatz, there exist
$\alpha, N\in \N$ such that
$$
  \alpha E_{V}^{N}\in (F_{1},\ldots, F_{s}, L^{\aff}) \subseteq  \Z[\bfU,\bfX]. 
$$
We use the effective version of this result~\cite[Theorem~2]{DKS} to
bound the integer~$\alpha$.

\begin{lemma}
  \label{lem:8} 
With notation as above, there exist $\alpha, N\in \N$ such that
$$
  \alpha E_{V}^{N}\in (F_{1},\ldots, F_{s}, L^{\aff}) \subseteq \Z[\bfU,\bfX]
$$ 
and
$$
    \log \alpha\le  A_1(m)d^{m+\min\{s,2m+1\}}h + A_2(m,s) d^{m+\min\{s,2m+2\}} 
$$
with
  \begin{equation*}
 \begin{split}
A_1(m)& =  10m  + 4,\\
A_2(m,s)& = (54m+98)\log(2m+5)  +24(m+1)\log \max\{1,s-2m\}.
 \end{split}
 \end{equation*}
\end{lemma}

\begin{proof}
The system of polynomials $F_{1},\ldots,F_{s}, L^{\aff}$ verifies the bounds
$$
  \deg F_{j} \le d, \quad \deg L^{\aff} =2, \quad \h(F_{j})\le h, \quad
  \h(L^{\aff})=0. 
$$
The case when $d=1$ can be easily treated applying Cramer's rule to
the  system of linear equations $F_{i}=0$, $i=1,\ldots, s$. Hence, we
assume that $d\ge2$. 

We apply~\cite[Theorem~2]{DKS} to the variety $ \A^{2m+1}_{\Q}$ and
the polynomials $E_{V}$, $F_{1},\ldots,F_{s}$ and $ L^{\aff}$.  From the
statement of~\cite[Theorem~2]{DKS}, we consider the
parameter $D$ and the sum over $\ell$ in the bound on $\alpha$, which
we denote by $\Sigma$. In our situation, the 
parameters  $n$ and $r$ in the notation of this theorem, are equal to $2m+1$. 

For $s+1\le 2m+2$, we have that $ D \le 2 d^s $ and $ D\Sigma \le
2sd^{s-1}h + d^sh\le (s+1)d^sh$ whereas, for $s+1> 2m+2$, we have that
$D \le d^{2m+2}$ and $ D\Sigma \le (2m+2)h d^{2m+1}$.  In either case,
$$
D \le  2d^{\min\{s, 2m+2\}}  \mand D\Sigma \le  (2m+2) d^{\min\{s, 2m+1\}} h.
$$
Thus, since
$\deg \A^{2m+1}_{\Q}= 1 $ and $\hcan (\A^{2m+1}_{\Q})=0$, it follows
that
\begin{align*}
  \log \alpha &\le 2D
\deg E_{V} \bigl(\frac{3\h(E_{V})}{2\deg E_{V}} +  \Sigma \\
     & \quad + \((12m+6)+17\)\log((2m+1)+4) 
     \\ & \quad +3(2m+2)\log(\max\{1,s-2m\})\bigr) \\
& \le 6 d^{\min\{s,2m+2\}} \h(E_{V}) + 2(2m+2) d^{\min\{s, 2m+1\}} h \deg E_{V}\\
& \qquad+  4d^{\min\{s, 2m+2\}}\deg E_{V}
 \bigl( (12m+23)\log(2m+5) 
 \\& \qquad\quad +6(m+1)\log \max\{1,s-2m\}\bigr)  .
\end{align*}
Applying Lemma~\ref{prop:3}, we obtain
\begin{align*}
  \log \alpha  & \le     6 d^{\min\{s,2m+2\}}  \(m d^{m-1}h+(m+1)d^{m}\log(m+1)\)\\
& \quad+ 2(2m+2) d^{m+\min\{s, 2m+1\}} h \\
& \qquad +  4d^{m+\min\{s, 2m+2\}} \bigl((12m+23)\log(2m+5) 
 \\ & \qquad \quad+6(m+1)\log \max\{1,s-2m\}\bigr).
\end{align*}
The  coefficient multiplying $h$ in the expression above can be
bounded~by
\begin{multline*}
 6  d^{m+\min\{s,2m+2\}-1} m  + 2 d^{m+\min\{s, 2m+1\}} (2m+2)\\
  \le  6 d^{m+\min\{s,2m+1\}} m  + 2d^{m+\min\{s, 2m+1\}}
 (2m+2) \\= A_1(m,s) d^{m+\min\{s,2m+1\}}.
\end{multline*}
By replacing  $\log(m+1)$ with $\log(2m+5)$ and after simple calculations,
 we obtain the desired expression for  $A_2(m,s)$. 
\end{proof}

We now recall the standard bound for the height of the composition of
polynomials with integer coefficients, see, for
instance,~\cite[Lemma~1.2(1.c)]{KPS}.

\begin{lemma}\label{hh}
Let $F\in\Z[Y_1,\ldots, Y_\ell]$ and $G_1,\ldots,
G_\ell\in\Z[\bfX]$. Set 
$$
d=\max_{i=1,\ldots,\ell}\deg G_i \mand h=\max_{i=1,\ldots,\ell}\h(G_i).
$$
Then
$$
  \h\(F(G_1,\ldots, G_\ell)\)\leq \h(F) +\deg F \(h
+\log(\ell+1)+  d \log(m+1)\).
$$
\end{lemma}

\begin{lemma}
  \label{lem:9}
  Let notation be as above. Then there exists $\beta\in\N$ such that
$$ 
  \log \beta\le B_1(m)d^{2m-1} h+ B_2(m)d^{2m}
$$
with
$$
B_1(m) =  2 m \mand 
B_2(m) =   (2m+4)\log(m+1) +  4m+2,
$$
 such that, if $p$ is a prime number not dividing $ \beta$, then the
 reduction of $E_{V}$ modulo $p$ is a squarefree polynomial of degree
 $T$ in the variable $U_{0}$.
\end{lemma}

\begin{proof} By Lemma~\ref{prop:3}, $\deg_{U_{0}}E_V=T$. Let
  $\beta_{0}$ be the coefficient of the monomial $U_{0}^{T}$ in
  $E_{V}$.  If $p\nmid \beta_{0}$, then reduction of $E_V$ modulo $p$
  has also degree $T$ in the variable $U_{0}$.

In addition, $E_{V}$ is squarefree and so 
$$
\Delta:=\Res_{U_{0}}\(E_{V},\frac{\partial E_{V}}{\partial U_{0}}\) \in \Z[U_{1},\ldots, U_{m}]
$$ 
is a nonzero polynomial. If $p$ does not divide one of the
nonzero coefficients of this polynomial, then $E_{V}\pmod{p}$ is also
squarefree. Thus we choose $\beta$ as the absolute value of $\beta_0$ 
times any nonzero coefficient of $\Delta$.

The logarithm of $|\beta_{0}|$ is bounded by the height of $E_{V}$. 
Hence, by Lemma~\ref{prop:3}, 
\begin{equation}
  \label{eq:beta0}
  \log|\beta_{0}|\le  m d^{m-1}h+(m+1)d^{m}\log(m+1).  
\end{equation}
By~\cite[Theorem~1.1]{Som04}, the Sylvester resultant of two generic
univariate polynomials of respective degrees $T$ and $T-1$, has $2T+1$ coefficients, degree $2T-1\le 2d^{m}-1$ and height bounded by $2T\log T\le
2md^{m}\log d$.  By Lemma~\ref{prop:3}, 
\begin{align*}
  \deg E_V , \deg \frac{\partial E_{V}}{\partial U_{0}} &\le  d^{m}, 
\\
  \h(E_V) , \h\(\frac{\partial E_{V}}{\partial U_{0}}\) & \le m
  d^{m-1}h+(m+1)d^{m}\log(m+1)+ m\log d.   
\end{align*}
Hence,  specializing this generic resultant in the coefficients of
$E_V$ and $\partial E_{V}/\partial U_{0}$, 
seen as polynomial in the variable $U_0$,
and using Lemma~\ref{hh} with $F = \Delta$, $\ell = 2T+1 \le 2d^{m}+1$ 
and $k = m$, we get
\begin{equation*}
\begin{split}
  \h(\Delta) & \le  
   2md^{m}\log d \\
   & \quad + (2d^{m}-1)\bigl( m d^{m-1}h+(m+1)d^{m}\log(m+1)+   m\log d\\ 
 & \quad  +
   \log (2d^{m}+2) + d^{m}\log (m+1)\bigr)\\
 & \le 
   2md^{m}\log d \\
   & \quad + (2d^{m}-1)\bigl( m d^{m-1}h+(m+2)d^{m}\log(m+1) +   m\log d\\ 
 & \quad  +
   \log (2d^{m}+2)\bigr) . 
     \end{split}
\end{equation*}
Taking into account that
$\log(2d^m+2) \le (m+1)d$, 
we get
  \begin{equation}
 \begin{split}
  \label{eq:h-Delta}
\h(\Delta)\le (2d^m-1) & \( md^{m-1}h + (m+2)d^{m}\log(m+1)\) \\ 
&\qquad \qquad \qquad \qquad \quad +
2d^{m} (2m+1). 
     \end{split}
\end{equation}
Adding~\eqref{eq:beta0} and~\eqref{eq:h-Delta},  we easily 
derive the stated result. 
\end{proof}

\subsection{Proof of Theorem~\ref{thm:1}}

We assume that $d \ge 2$ as otherwise the result is trivial by the
Hadamard bound on the determinant 
of the corresponding system of
linear equations.

Set $\fA=\alpha\beta$ with $\alpha$ as in Lemma~\ref{lem:8} and
$\beta$ as in Lemma~\ref{lem:9}. 
If $p\nmid \fA$, then $p\nmid \beta$ and, by Lemma~\ref{lem:9}, the reduction of
the eliminant $E_V$ modulo $p$ is a squarefree polynomial of degree
$T$ in the variable $U_{0}$. 

Recall that $\ov \cV_{p}$ denotes the fibre of the
scheme $\ov \cV$ over the prime $p$. This is a subscheme of
$\P^{m}_{\F_{p}} $.  From the definition of the eliminant ideal, we
can see that $\Elim(I(\ov \cV_{p}))$ divides $E_V \pmod{p}$.  Since
this latter polynomial is squarefree, it follows that $ \Elim(I(\ov
\cV_{p})) $ is squarefree too.

By Lemma~\ref{prop:1}\eqref{item:3}, this implies that the
subcheme $\ov \cV_{p}$ is reduced and, by 
Lemma~\ref{lem:3}\eqref{item:5}, it is of degree $T$. Applying
Lemma~\ref{prop:1}\eqref{item:3} again, we deduce that 
$  \Elim(I(\ov \cV_{p}))$ has degree $T$ and so 
$$
  \Elim(I(\ov \cV_{p})) \equiv \lambda E_{V} \pmod{p}
$$
with $\lambda\in \F_{p}^{\times}$.  By Lemma~\ref{lem:9}, the
polynomial $E_{V} \pmod{p}$ has degree $T$ in the variable
$U_{0}$. This implies that the subscheme $\ov \cV_{p}$ is contained in
the open subset $\A_{\F_{p}}^{m}$. Hence, $\ov \cV_{p}$ is a
subvariety of degree $T$ which is contained in $V_{p}$.

If $p$
is a prime not dividing $\fA$, then $p\nmid \alpha $ and so $\alpha$ is
invertible modulo $p$.  Write $L_{p}^{\aff}\in \F_{p}[\bfU,\bfX]$ for
the reduction modulo $p$ of $L^{\aff}$. Then
$$
 E_{V}^{N} \pmod{p} \in (F_{1,p},\ldots, F_{m,p}, L_{p}^{\aff}) \subseteq \F_{p}[\bfU,\bfX]  
$$
with $F_{1,p}, \ldots, F_{s,p}$ as in~\eqref{eq:Fip}.
Write
\begin{equation}
  \label{eq:13}
  E_{V}^{N} \pmod{p} = A L_{p}^{\aff} + \sum_{j=1}^{m}B_{j}F_{j,p}   
\end{equation}
with $A, B_{j}\in \F_{p}[\bfU,\bfX]$. Let $\bfxi$ be a zero of
$F_{j,p}$, $j=1,\ldots, s$, in $\ov \F_{p}^{m}$. Evaluating the
equality~\eqref{eq:13} at this point, we obtain
$$
  E_{V}^{N}(\bfU) \pmod{p} = A(\bfU,\bfxi) L^{\aff}(\bfU, \bfxi). 
$$
It follows that $L_{p}^{\aff}(\bfU, \bfxi)$ divides $E_{V}(\bfU)
\pmod{p} $ for every such point. Since for every pair of distinct
points $\bfxi_1$ and $\bfxi_2$ in $\ov \F_{p}^{m}$, the linear forms
$L_{p}^{\aff}(\bfU, \bfxi_1)$ and $L_{p}^{\aff}(\bfU, \bfxi_2)$ are
coprime, we conclude that the zero set of $F_1,\ldots,F_s$ in $\ov
\F_{p}^{m}$ has at most $ \deg E_{V}=T$ points.  Hence $V_{p}$ is of
dimension 0 and degree $T$, as stated.

The bound for $\fA$ follows from the bounds for $\alpha$ in
Lemma~\ref{lem:8} and for $\beta$ in Lemma~\ref{lem:9}. Indeed, with
the notation therein, the quantity
$A_1(m)d^{m+\min\{s,2m+1\}}+B_1(m)d^{2m-1}$ can be bounded by
  $$
   (10m  +
4)d^{m+\min\{s,2m+1\}}  + 2 m d^{2m-1} \le  (11m  + 4) d^{3m +1}  =
C_1(m) d^{3m+1}, 
 $$
 and $A_2(m,s)d^{m+\min\{s,2m+2\}}+B_2(m)d^{2m-1}$ can
 be bounded by
   \begin{equation*}
 \begin{split}   
&  \bigl((54m+98)\log(2m+5) \\
& \qquad  +24(m+1)\log \max\{1,s-2m\} \bigr)d^{m+\min\{s,2m+2\}}\\
& \qquad \qquad +\((2m+4)\log(m+1) +  (4m+2)\)d^{2m}\\
& \qquad\le    
\bigl((54m+98)\log(2m+5)  +24(m+1)\log \max\{1,s-2m\}\\ 
& \qquad \quad + \frac{1}{8}
\((2m+4)\log(m+1) +  (4m+2)\) \bigr) d^{3m+2} \\
& \qquad \le(55 m +99) \log ((2m+5)s)  d^{3m+2}   =  d^{3m+2} C_2(m,s)
 \end{split}
 \end{equation*}
with $C_{1}(m)= 11m+4$ and $C_{2}(m,s)=  (55 m +99) \log ((2m+5)s)$. 
Hence 
$$
  \log \fA \le  C_1(m) d^{3m +1} h +C_2(m,s)d^{3m +2} ,
$$
concluding the proof. 

\section{Bounds for the Degrees and the Heights of Products and
  Compositions of Rational Functions} \label{sec:polyn-syst-height}


In this section, 
we collect several useful bounds on the height of various polynomials
and rational functions. These lemmas  are used in the proof of our
results in~\S\ref{sec:Period}, \S\ref{sec:OrbVar GenAvoid} 
and~\S\ref{sec:OrbVar UML}, and some of them can be of independent
interest.


The following bound on the height of a product of polynomials, which  follows 
from~\cite[Lemma~1.2]{KPS}, underlines our estimates. 

\begin{lemma}
\label{lem:HeightFact}
Let $F_1, \ldots, F_s\in \Z[\bfX]$.  Then 
\begin{align*}
- 2  \sum_{i=1}^s\deg F_{i} \log(m+1) \le \h\(\prod_{i=1}^sF_{i}\)
& -\sum_{i=1}^s \h(F_i)\\
& \le   \sum_{i=1}^s\deg F_{i}  \log(m+1). 
\end{align*}
\end{lemma}

We also frequently use the trivial bound on the height of a sum of polynomials
\begin{equation}
\label{eq:HeightSum}
\h\(\sum_{i=1}^s F_i\) \le \max_{i=1, \ldots, s} \h(F_i) +\log s.
\end{equation}

We already used the bound for the composition of polynomials
(see Lemma~\ref{hh}). We now specialize it to polynomials with equal number of
variables.

\begin{lemma}\label{lem:HeighCompos-Poly}
  Let $F, G_1,\ldots, G_m\in\Z[\bfX]$. Set $d=\max_{i}\deg
  G_i$ and $ h=\max_{i}\h(G_i)$. Then
\begin{align*}
& \deg F(G_1,\ldots, G_m)\le d \deg F, \\
&  \h\(F(G_1,\ldots, G_m)\) \leq \h(F) +h \deg F+ 
 (d+1) \deg F \log(m+1).
\end{align*}
\end{lemma}

The following is and extension of  Lemma~\ref{lem:HeighCompos-Poly} 
to the composition of rational
functions.

\begin{lemma}\label{lem:HeighCompos}
  Let $R, S_1,\ldots, S_m\in\Q(\bfX)$ such that the
  composition $R(S_{1},\ldots, S_{m})$ is well defined. Set $
  d=\max_{i}\deg S_i$ and $ h=\max_{i}\h(S_i)$. Then
\begin{align*}
\deg R(S_{1},\ldots,
S_{m})& \le dm \deg R, \\
  \h\(R(S_{1},\ldots,
S_{m})\)&\leq \h(R) +h \deg R+ 
 (3dm+1) \deg R \log(m+1).
\end{align*}
\end{lemma}

\begin{proof}
Let $R=P/Q$ with coprime $P,Q\in \Z[\bfX]$ and write 
$$
  P=\sum_{\bfa}\alpha_{\bfa}\bfX^{\bfa} \mand Q=\sum_{\bfa}\beta_{\bfa}\bfX^{\bfa}
$$
with $\alpha_{\bfa},\beta_{\bfb}\in \Z$. We suppose for simplicity
that
$$
D=\deg P\ge \deg Q,  
$$
since the other case can be reduced to
this one by considering the inverse $R^{-1}$. 

Let also $S_{i}=F_{i}/G_{i}$
with  coprime $F_{i},G_{i}\in \Z[\bfX]$. 
Consider the polynomials
$$
B=\prod_{j}G_{j} \mand A_{i}=F_{i}\prod_{j\ne i}G_{j}
$$
and set $\bfA=(A_{1},\ldots, A_{m})$. 
Then  $R(S_{1},\ldots, S_{m})=
U/V$ with
$$
  U= \sum_{\bfa}\alpha_{\bfa} B^{D-|\bfa|}\bfA^{\bfa} \mand V= \sum_{\bfa}\beta_{\bfa}B^{D-|\bfa|}\bfA^{\bfa}.
$$

By Lemma~\ref{lem:HeightFact}, for each $\bfa$ with $|\bfa|\le D$, 
$$
\deg(B^{D-|\bfa|}\bfA^{\bfa})  \le mDd, \quad   \h( B^{D-|\bfa|}\bfA^{\bfa})   \le mDh + mDd\log(m+1).    
$$
Hence $\deg U, \deg V \le mDd$, which gives the degree bound for
the rational function 
$R(S_{1},\ldots, S_{m})$. For the height bound, we have that 
\begin{equation} \label{eq:12}
  \begin{split}
  \h(U) & \le h(P)+ mDh + mDd\log(m+1) + \log\binom{D+m}{m}\\
 & \le h(R)+ mDh + mDd\log(m+1) + D\log(m+1),     
\end{split}
\end{equation}
and similarly for $V$. 

Let $\widetilde U, \widetilde V\in \Z[\bfX]$ coprime with $\widetilde U/\widetilde V =U/V$. Then
$\widetilde U\mid U$ and $\widetilde V\mid V$. Then, by Lemma~\ref{lem:HeightFact}, 
\begin{equation}\label{eq:14}
  \h(\widetilde U) \le \h(U) + 2\log(m+1) \deg U.
\end{equation}
and similarly for $\widetilde V$. From~\eqref{eq:12} and~\eqref{eq:14}, it
follows that
\begin{align*}
  \h(\widetilde U)  &\le  h(R)+ mDh + D(md+1)\log(m+1) + 
  2mDd\log(m+1)\\
&\le  h(R)+ mDh + D(3md+1)\log(m+1),
\end{align*}
and similarly for $\widetilde V$, 
which gives the bound for the height of the composition. 
\end{proof}

We now use Lemma~\ref{lem:HeighCompos-Poly} to bound 
the degree and height of iterations of polynomial systems.

\begin{lemma}
\label{lem:HeighIter-Poly}
Let $F_1,\ldots,F_m \in \Z[\bfX]$ be polynomials 
of degree at most $d\geq2$ and  height at most $h$. Then, for any positive integer $k$,
the polynomials  $F_1^{(k)}, \ldots ,F_m^{(k)}$ defined by~\eqref{eq:RatIter},
are of degree at most 
$d^k$
and of  height at most
$$
h \frac{d^k-1}{d-1}+d(d+1)\frac{d^{k-1}-1}{d-1}\log(m+1). 
$$
\end{lemma}

\begin{proof} The bound on the degree is trivial, and the inequality for the   height also follows straightforward by induction on the number of iterates $k$. Indeed, for $k=1$ we have equality by definition. Suppose the statement true for the first $k-1$ iterates. For every $i=1,\ldots,m$, we apply Lemma~\ref{lem:HeighCompos-Poly} to the polynomial 
$$
F_i^{(k)}=F_i^{(k-1)}\(F_1,\ldots,F_m\)
$$
and we get that the height of this polynomial is bounded by
\begin{equation*}
\begin{split}
\h(F_i^{(k-1)})+ (h + (d+1)&\log(m+1)) \deg F_i^{(k-1)} \\
&\le h \frac{d^{k-1}-1}{d-1}+d(d+1)\frac{d^{k-2}-1}{d-1}\log(m+1)\\
& \qquad \qquad \qquad +  (h + (d+1)\log(m+1)) d^{k-1}
\\
&\le h \frac{d^k-1}{d-1}+d(d+1)\frac{d^{k-1}-1}{d-1}\log(m+1), 
\end{split}
\end{equation*}
which concludes the proof. 
\end{proof}

For rational functions we apply Lemma~\ref{lem:HeighCompos} to derive 
a similar result.

\begin{lemma}
\label{lem:HeighIter}
Let $R_1,\ldots,R_m \in \Q(\bfX)$  be rational functions 
of degree at most $d$ and  height at most $h$. If either $d\geq2$ or $m\geq2$ then, for any positive integer $k$,
the rational functions  $R_1^{(k)}, \ldots ,R_m^{(k)}$ defined by~\eqref{eq:RatIter},
are of degree at most $d^k m^{k-1},$
and of  height at most
$$\(1+d\frac{d^{k-1} m^{k-1}-1}{dm-1}\)h+ 
d(3dm+1)\frac{d^{k-1} m^{k-1}-1}{dm-1}\log(m+1).$$
\end{lemma}

\begin{proof} 
  The bound for the degree follows easily from Lemma~\ref{lem:HeighCompos}.  
  We prove the bound for the height by
  induction on $k$. For $k=1$ the bound is trivial. For $k\ge 2$, we
  assume that the bound holds for the first $k-1$ iterates.

  Applying  Lemma~\ref{lem:HeighCompos} with $R_i$ and
  $R_i^{(k-1)}$, $i=1,\ldots,m$, and the induction hypothesis, we
  obtain that $\h(R_i^{(k)}) $ is bounded by 
\begin{equation*}
\begin{split}
  \h(R_i^{(k-1)}&)+h \deg(R_i^{(k-1)}) +(3dm+1)\deg(R_i^{(k-1)})\log(m+1)\\
& \le   \(1+d\frac{d^{k-2} m^{k-2}-1}{dm-1}\) h\\
 & \qquad  + d(3dm+1)\frac{d^{k-2} m^{k-2}-1}{dm-1}\log(m+1)\\
&\qquad \qquad +h d^{k-1}m^{k-2}+(3dm+1)d^{k-1}m^{k-2}\log(m+1)
\\
&=   \(1+d\frac{d^{k-1} m^{k-1}-1}{dm-1}\) h \\
&\qquad+ 
d(3dm+1)\frac{d^{k-1} m^{k-1}-1}{dm-1}\log(m+1), 
\end{split}
\end{equation*}
where we have used the identity
$$
d\frac{d^{k-2} m^{k-2}-1}{dm-1} + d^{k-1}m^{k-2} = 
d\frac{d^{k-1} m^{k-1}-1}{dm-1}.
$$
\end{proof}

\section{Periodic Points}
\label{sec:Period}

 \subsection{Definitions and main results}
 \label{sec:def res per}

We start with the following standard definition of $k$-periodicity.

\begin{definition}
  \label{def:Per k} Let $K$ be a field and $\bfR\in K(\bfX)^{m}$  a
  system of rational functions as in~\eqref{eq:16}. Given $k\ge1$, 
we say that $\bfw \in \ov K^m$ is \emph{$k$-periodic} if the element 
$\bfw_k$ exists in the orbit~\eqref{eq:Orb R} and we have
$\bfw_k = \bfw_0$. 
\end{definition}

In this definition, we do not request that $k$ is the smallest integer
with this property. On the other hand, this notion of $k$-periodicity
is more restrictive than the condition
$\bfR^{(k)}(\bfw)=\bfw_0$, see the discussion after \eqref{eq:Orb R}. 

%

We first prove the following  result for systems of rational functions.

\begin{theorem}
  \label{thm:Cycles} Let $m,d\in\N$ with $d,m\ge 2$, and $\bfR=(R_1,\ldots,R_m)$ be a
  system of $m$ rational functions in $\Q(\bfX)$ of degree
  at most $d$ and of height at most $h$.  Assume that $\bfR$ has
  finitely many periodic points of order $k$ over $\C$.  Then there
  exists an integer $\fA_k \ge 1$ with
$$
    \log \fA_k \ll_{d,h,m} (dm)^{k(3m+5)} 
$$
such that, if $p$ is a prime number not dividing $\fA_k$, then the
reduction of $\bfR$ modulo $p$ has at most $(2m^{k}d^k)^{m+1}$
periodic points of order $k$.
\end{theorem}

In the particular case of polynomials, the bound of
Theorem~\ref{thm:Cycles} simplifies as follows:

\begin{theorem}
\label{thm:Cycles-Poly} 
Let $d,m\ge 2$, and $\bfF = (F_1,\ldots,F_m)$ be a system of $m$
polynomials in $ \Z[\bfX]$ of degree at most $d$ and of
height at most $h$.  Assume that $\bfF$ has
finitely many periodic points of order $k$ over $\C$.  Then there
exists an integer $\fA_k \ge 1$ with
$$
\log \fA_k \ll_{d,h,m}  
 d^{k(3m+2)} 
$$
such that, if $p$ is a prime number not dividing $\fA_k$, then the
reduction of $\bfF$ modulo $p$ has at most $d^{km}$ periodic points of
order $k$.
\end{theorem}

Using these theorems, it is possible to recover some of the
results of Silverman and of Akbary and Ghioca, that give lower bounds
on the period length which are roughly of order $\log \log p$ for all
primes~$p$~\cite[Corollary 12]{Silv2}, and of order $\log p$ for
almost all of them~\cite[Theorem~1.1(1)]{AkGh}, 
see~\cite[Corollaries~2.3 and~2.4]{CDOSS}.

Similarly, Theorems~\ref{thm:Cycles} and~\ref{thm:Cycles-Poly} can be
used with $k$ of order $\log p$ and $\log \log p$ for almost all and
all primes $p$, respectively, to get nontrivial upper bounds on the
number of periodic points of order $k$ (or even at most $k$).
%
%
%
%
%

\subsection{Proof of Theorem~\ref{thm:Cycles}}

The result is a direct consequence of Theorem~\ref{thm:1} and
Lemma~\ref{lem:HeighIter}. Indeed, let $\bfR^{(k)}$ be the iteration of the system of rational functions
$\bfR$ as in~\eqref{eq:RatIter}. As in~\eqref{eq:RedForm},  write
$$R_i^{(k)}= \frac{F_{i,k}}{G_{i,k}}
$$
with coprime
$F_{i,k}, G_{i,k}\in\Z[\bfX]$ and $G_{i,k}\ne 0$, and consider then the  system of equations
$$F_{i,k}-X_iG_{i,k} =0, \qquad i =1, \ldots, m.
$$

From the solutions to this system of equations, we have to extract
those that come from the poles of $R_i^{(j)},\,j\leq k$, that is, from
the zeroes of $\prod_{i=1}^m\prod_{j=1}^k G_{i,j}$.  For this we
introduce a new variable $X_0$, and thus, the set of $k$-periodic
points of $\bfR$ coincides with the zero set
 $$
 V_k=Z\(F_{1,k}-X_1G_{1,k},\ldots,F_{m,k}-X_mG_{m,k},1-X_0\prod_{i=1}^m\prod_{j=1}^k G_{i,j}\).
 $$
By Lemma~\ref{lem:HeighIter} and the fact that $dm\geq2:$
\begin{equation}
\label{eq:Deg G}
  \deg \(X_0\prod_{i=1}^m\prod_{j=1}^k G_{i,j}\)\leq 1+m\sum_{j=1}^kd^jm^{j-1}\le 2 (dm)^k.
\end{equation}
Further, by Lemmas~\ref{lem:HeightFact} and~\ref{lem:HeighIter}, 
 \begin{align*}
 \h&\(X_0\prod_{i=1}^m\prod_{j=1}^k G_{i,j}\) =  \h\(\prod_{i=1}^m\prod_{j=1}^k G_{i,j}\)\\
  & \qquad \le 
2(dm)^{k}\, \log(m+1)  + \sum_{i=1}^m \sum_{j=1}^k\h\(G_{i,j}\)  \\
 & \qquad \le  2(dm)^{k}\, \log(m+1)  + m\biggl(\sum_{j=1}^k\(1+d\frac{d^{j-1} m^{j-1}-1}{dm-1}\) h \\
 & \qquad\qquad \qquad  \qquad \qquad+ 
d(3dm+1)\frac{d^{j-1} m^{j-1}-1}{dm-1}\log(m+1)\biggr)\\ 
& \qquad\le 2(dm)^{k}\, \log(m+1)  + m\biggl(4d(dm)^{k-2}h\\
 & \qquad\qquad  \qquad \qquad  \qquad+ 2d(3dm+1)(dm)^{k-1}\log(m+1)\biggr).
\end{align*}
Hence 
 $$
\h\(X_0\prod_{i=1}^m\prod_{j=1}^k G_{i,j}\) \ll_{d,h,m}(dm)^k.
$$

Also, for every $i=1,\ldots,m$, we easily see that Lemma~\ref{lem:HeighIter}
and the bound~\eqref{eq:HeightSum} yield
$$
\deg \(F_{i,k}-X_iG_{i,k}\) \le d^k m^{k-1}+1, 
$$
and 
 \begin{align*}
 \h\(F_{i,k}-X_iG_{i,k}\)&\le \h\( R_i^{(k)}\) + \log 2 \\
&   \le \(1+d\frac{d^{k-1} m^{k-1}-1}{dm-1}\) h \\
 &  \qquad + 
d(3dm+1)\frac{d^{k-1} m^{k-1}-1}{dm-1}\log(m+1) + \log 2.
\end{align*}
Hence
$$
\h\(F_{i,k}-X_iG_{i,k}\)  \ll_{d,h,m}d^{k}m^{k-1}. 
$$
 We apply now Theorem~\ref{thm:1} (with $s = m+1$ polynomials and  $m+1$ variables) and 
derive
$$
\log \fA_k \ll_{d,h,m}  (dm)^{k+k(3(m+1)+1)} h+ (dm)^{k(3(m+1)+2)} \ll_{d,h,m}    (dm)^{k(3m+5)} . 
$$

Next, we  denote by $N_k$ the number of points of 
$V_k$ over $\C$, which is equal to 
the number of periodic points of order $k$ of $R_1,\ldots,R_m$ over $\C$.  
Using the degree bounds~\eqref{eq:Deg G},  by  B{\'e}zout theorem
we obtain $$N_k\le 2(md)^k\(m^{k}d^{k}+1\)^{m}\leq(2m^kd^k)^{m+1},$$ which yields the desired bound.

\subsection{Proof of  Theorem~\ref{thm:Cycles-Poly}}

As in the proof of Theorem~\ref{thm:Cycles}, the result is an
immediate consequence of Theorem~\ref{thm:1} and
Lemma~\ref{lem:HeighIter-Poly}.  Indeed, we apply Theorem~\ref{thm:1}
with
$$
V_k=Z(F_1^{(k)}-X_1,\ldots,F_m^{(k)}-X_m),
$$
getting, after simple calculations, that 
$$
\log \fA_k \ll_{d,h,m} 
d^{k+k(3m+1)} h+ d^{k(3m+2)} \ll_{d,h,m} d^{k(3m+2)}. 
$$
We now denote by $N_k$ the number of points of $V_k$ over $\C$, which
is equal to the number of periodic points of order $k$ of
$F_1,\ldots,F_m$ over $\C$.  Using Lemma~\ref{lem:HeighIter-Poly} and
the fact that $N_k\le d^{km}$, we obtain immediately the desired
bound.

%

\subsection{Lower bounds on the number of $k$-periodic points}

The bound on the $k$-periodic points given by
Theorem~\ref{thm:Cycles-Poly} is tight for some particular polynomial
systems. Indeed, let $d\ge 0$ and consider the system
$\bfF=(F_1,\ldots, F_m)$ with $F_{i}=X_{i}^{d}$ .  For $k\ge 1$, the
$k$-th iterate is given by $ F_{i}^{(k)}=X_{i}^{d^{k}}$, $i=1,\ldots,
m.$ A $k$-periodic point is a solution to the system
\begin{equation}
  \label{eq:15}
  X_{i}^{d^{k}}-X_i=0,  \quad 
i=1,\ldots, m. 
\end{equation}
This system of equations has a finite number of solution over the
complex numbers.  Set $\fA=d^{k}-1$. If $p$ is a prime not dividing
$\fA$, then the system of equations~\eqref{eq:15} has exactly $d^{km}$
solutions in $\ov \F_{p}^{m}$. Hence, the reduction of $\bfF$ modulo
$p$ has exactly $d^{km}$ periodic points of order $k$.



\section{Iterations Generically Escaping a Variety}
\label{sec:OrbVar GenAvoid}

\subsection{Problem formulation and definitions}
\label{sec:orbit gen}

We  next   study of the frequency of the orbit intersections
of two rational function systems.  In the univariate
case, Ghioca, Tucker and Zieve~\cite{GTZ1,GTZ2} have proved that, if
two univariate nonlinear complex polynomials have an infinite
intersection of their orbits, then they have a common iterate. 
No  results of this kind are known for arbitrary rational
functions.  

The analogue of this result by Ghioca, Tucker and Zieve~\cite{GTZ1,GTZ2}  cannot hold
over finite fields.  Instead, we obtain an upper bound for the
frequency of the orbit intersections  of a rational 
function system. More
generally, we bound the number of points in such an orbit that belong
to a given algebraic variety.

As before, we first obtain results for general systems of rational 
functions and polynomials, and we then obtain stronger bounds 
for systems of the form~\eqref{eq:Polysyst}.

Let $K$ be a field and 
$$
\bfR = (R_1,\ldots,R_m), \qquad R_1,\ldots,R_m \in K(\bfX)
$$
a system of $m$ rational functions in $m$ variables over $K$ as in
\eqref{eq:16}. For $n\ge 1$, we denote by $\bfR^{(n)}$ the $n$-th
iteration of this system, as long as this iteration is well-defined. 

Given an initial point $\bfw\in \ov K^{m}$, we consider the sequence
given by 
$$
 \bfw_{0}=\bfw \mand \bfw_{n}=\bfR(\bfw_{n-1}) \text{ for } n\ge 1,
$$
as in~\eqref{eq:Orb R}.  As discussed after \eqref{eq:Orb R}, this
sequence terminates when $\bfw_n$ is a pole of the system $\bfR$.
Recall that the orbit of $\bfw$ is the the subset
$\Orb_\bfR(\bfw)=\{\bfw_{n}\mid n\ge1\}\subset \ov K$. We put
\begin{equation}
 \label{def:T(w)}
T(\bfw)=   \# \Orb_\bfR(\bfw)\in \N\cup \{\infty\}.
\end{equation}




Now let $K=\Q$ and, for $n\ge 1$, write 
$$
R_i^{(n)}= \frac{F_{i,n}}{G_{i,n}}
$$
with coprime
$F_{i,n}, G_{i,n}\in\Z[\bfX]$ and $G_{i,n}\ne 0$, 
as in~\eqref{eq:RedForm}.
Given a prime $p$ such that  $G_{i,j}\not\equiv0 \pmod{p}$,
$j=1,\ldots, n$, we can
consider the reduction modulo $p$ of the iteration $\bfR^{(n)}$. We 
denote it by
$$
  \bfR_{p}^{(n)}=(R_{1,p}^{(n)},\ldots,R_{m,p}^{(n)})\in \F_{p}(\bfX)^{m}.
$$

Let $V\subset \A^{m}_{\Q}$ be the affine algebraic variety over $\Q$ defined
by a system of polynomials $P_i\in\Z[\bfX]$,
$i=1,\ldots,s$. For a prime $p$, we denote by
$V_{p}\subset\A^{m}_{\F_{p}}$ the variety over $\F_{p}$ defined by the
reduction modulo $p$ of the system $P_{i}$, $i=1,\ldots, s$.  

Let $\bfw\in\ov\F_p^m$ be an initial point, $N\in
\N$, and suppose that $G_{i,j}\not \equiv 0 \pmod{p}$, $j=0,\ldots,
N-1$. We then define
$$
\fV_{\bfw}(\bfR,V;p,N)=\left\{ n\in \{0,\ldots, N-1\}  \mid  
  \bfR_p^{(n)}(\bfw)\in V_p(\ov \F_{p})\right\}.
$$ 
Namely, this is the set of values of $n\in \{0,\ldots, N-1\}$ such that
the iterate $ \bfR_p^{(n)}(\bfw)$ is defined and lies in the set $
V_p(\ov \F_{p})$.  One of our goals is obtaining upper bounds on
$\#\fV_{\bfw}(\bfR,V;p,N)$ that are uniform in~$\bfw$.
%


We now define the following 
class of pairs $(\bfR, V)$ of systems of rational functions and varieties:


\begin{definition}  \label{def:2}
 With notation as above, we say that the  iterations of $\bfR$ \emph{generically escape}
$V$ if, for every integer $k\ge 1$, the $k$-th iteration of
     $\bfR$ is well-defined and the set
$$
\{\bfw \in \C^m \mid \(\bfw,\bfR^{(k)}(\bfw)\)\in V(\C)\times V(\C)\}      
$$
is finite.
\end{definition}
    

We expect that this property of generic escape
 is satisfied for a
``random'' pair $(\bfR,V)$ consisting of a system and a variety of
dimension at most $m/2$.

We consider now two rational function systems $\bfR, \bfQ\in
\Q(\bfX)^{m}$.  For $N\in \N$, let $p$ be a prime such that that the
iterations $\bfR^{(j)}$ and $\bfQ^{(j)}$, $j=0,\ldots, N-1$, can be
reduced modulo $p$.  For $\bfu,\bfv  \in \ov \F_p^m$,  we define
$$
\fI_{\bfu,\bfv }(\bfR,\bfQ;p,N)=\left\{n \in \{0,\ldots, N-1\}
\mid \bfR_{p}^{(n)}(\bfu)=\bfQ_{p}^{(n)}(\bfv )\right\}.
$$

To bound the cardinality of this set, we introduce the following
analogue of Definition~\ref{def:2}:


\begin{definition}
  \label{def:3} 
  Let $\bfR, \bfQ\in \Q(\bfX)^{m}$.  We say that the iterations of  $\bfR$ and $\bfQ$
 \emph{generically escape each other\/} if, for every
  $k\in\N$, the $k$-th iterations of $\bfR$ and $\bfQ$ are
  well-defined and the set 
$$
\{\bfw \in \C^m \mid \bfR^{(k)}(\bfw) =
  \bfQ^{(k)}(\bfw)\}
$$ 
is finite.
\end{definition}

\subsection{Systems of rational functions}
\label{sec:orbit var rat fun}

We present our results in a simplified form where  all 
constants depend 
on subsets of the following vector of parameters  
\begin{equation}
\label{eq:rho}
\bfrho = (d, D, h, H, m, s).  
\end{equation} 
Consequently, in our results we use the notation  `$O_\bfrho$' and
`$\ll_\bfrho$', meaning that the implied constants do not  depend on
the  parameters $\varepsilon$ and~$N$. 

We also recall the definition of $T(\bfw)$ given by~\eqref{def:T(w)}.

\begin{theorem} 
\label{thm:OrbIntVar} 
Let $\bfR=(R_1,\ldots,R_m)$ be a system of $m\ge 2$ rational functions
in $\Q(\bfX)$ of degree at most $d\ge 2$ and of height at
most $h$.  Let $P_1,\ldots,P_s\in\Z[\bfX]$ of degree at most
$D$ and height at most $H$, and denote by $V\subset \A^{m}_{\Q}$ the
variety defined by this system of polynomials.  Assume that the iterations
of $\bfR$ generically escape $V$.  Then, there is a constant
$c(\bfrho)> 0$ such that for any real $\varepsilon >0$ and   $N\in \N$ with
\begin{equation}
\label{eq:eps rat}
N\ge \exp\(\frac{c(\bfrho) }{\varepsilon}\),
\end{equation} 
 there exists $\fB \in \N$ with
$$
\log \fB \le \exp\(
\frac{c({\bfrho})}{\varepsilon}\)
$$
such that, if $p$ is a prime number not dividing $\fB$, then
for any $\bfw \in \ov \F_p^m$ with $T(\bfw) \ge N$,
$$
\frac{\# \fV_{\bfw}(\bfR,V;p,N)  }{N}
 \le \varepsilon.
$$
\end{theorem}

We derive from Theorem~\ref{thm:OrbIntVar} the following bound for the
number of orbit intersection for two systems of rational functions. 

\begin{cor} 
\label{cor:OrbInt} 
Let $\bfR=(R_{1},\ldots,R_{m})$ and $\bfQ=(Q_{1},\ldots,Q_{m})$ be two
systems  of rational functions in $\Q(\bfX)$ of degree at most
$d$ and of height at most $h$ such that their iterations generically escape 
each other. 
Then there is a  constant $c(\bfrho) > 0$ such that, for
any real $\varepsilon >0$ and  $N\in \N$ with
$$
N\ge \exp\( \frac{c(\bfrho) }{\varepsilon}\),
$$   
there exists  $\fB \in \N$ with
$$
\log \fB \le \exp\(
\frac{c({\bfrho})}{\varepsilon}\)
$$
such that, if $p$ is a prime number not dividing $\fB$, then 
for any $\bfu,\bfv  \in
\ov \F_p^m$ with $T(\bfu) , T(\bfv) \ge N$,
$$
\frac{\# \fI_{\bfu,\bfv}(\bfR,\bfQ;p,N)  }{N}
 \le \varepsilon.
$$
\end{cor}

\begin{rem}
\label{rem: eps and p}
Alternatively, the bounds in  Theorem~\ref{thm:OrbIntVar} and
Corollary~\ref{cor:OrbInt} can be formulated taking $\varepsilon$ as a function of $p$. More precisely,  for some constant $c_0(\bfrho)>0$,
one can take
\begin{itemize}
\item $\varepsilon = c_0(\bfrho) /\log \log p$ for any prime $p$ and eliminate any 
influence of $\fB$. Since $\fB \le  \exp\exp(c(\bfrho) \varepsilon^{-1})$, the condition 
$p\nmid \fB$ is automatically satisfied for such $\varepsilon$,
provided that 
       $c_0(\bfrho)$ is large enough;
\item $\varepsilon = c_0(\bfrho)/\log Q$ for all but $o( Q/\log Q)$
  primes $p\le Q$, since $\fB$ has at most $ \log \fB \le
  \exp({c({\bfrho})}{\varepsilon^{-1}})$ prime divisors.
\end{itemize}
\end{rem} 

\begin{rem}
\label{rem: slow} 
For systems with slower than generic  growth of the
degree and height the bounds in  Theorem~\ref{thm:OrbIntVar} and
Corollary~\ref{cor:OrbInt} can be improved. Some examples of such 
systems are given in \S\ref{sec:rem}. \end{rem}

\subsection{Preparation}

We need the following simple combinatorial statement.

\begin{lemma}
\label{lem:Combin}
Let $2 \le M<N/2$.  For any sequence 
$$
0 \le n_1< \ldots< n_M\le N,  
$$
there exists $r\le 2N/(M-1)$ such that $n_{i+1} - n_i = r$ for at
least $(M-1)^2/4N$ values of $i \in \{1, \ldots, M-1\}$.
\end{lemma}

\begin{proof}
We denote by $I(s)$ the number of $i=1, \ldots, M-1$ 
with $n_{i+1} - n_i = s$. 
Clearly 
$$
\sum_{s=1}^{N} I(s) = M-1 \mand
\sum_{s=1}^{N} I(s)s = n_{M}- n_1 \le N.
$$
Thus, for any integer $t \ge 1$ we have 
\begin{equation*}
\begin{split}
\sum_{s=1}^{t} & I(s) = M-1 - \sum_{s=t+1}^N I(h) \\
&\ge   M-1 - \frac{1}{t+1}\sum_{s=t+1}^N I(s)s
\ge  M-1 - \frac{1}{t+1}N.
\end{split}
\end{equation*}
Hence, there exists $r \in \{1, \ldots, t\}$ 
with 
\begin{equation}
\label{eq:I(r)}
I(r) \ge \frac{1}{t} \sum_{s=1}^{t}  I(s) \ge  \frac{M-1  -  N/(t+1)}{t}.
\end{equation}
We now set 
$t =\fl{2N/(M-1)}$.  
Clearly 
$$
1 \le t \le \frac{2N}{(M-1)} \mand  \frac{N}{t+1} < \frac{M-1}{2}.
$$
Hence
$$
\frac{M-1  -  N/(t+1)}{t}\ge \frac{M-1}{2t} \ge \frac{(M-1)^2}{4N}, 
$$
which together with~\eqref{eq:I(r)} concludes the proof.
\end{proof}

\subsection{Proof of Theorem~\ref{thm:OrbIntVar}}
Let $p$ be a prime and $n\in\N$.  As at the beginning of this section,
we denote by $\bfR_p^{(n)}$ and $ V_p$ the reduction modulo $p$ of
$\bfR^{(n)}$ and $V$, respectively.  Fix an initial point $\bfw \in
\ov \F_p^m$ and let $M\in \N$ be the number of values of $n\in
\{0,\ldots, N-1\}$ such that $\bfR_p^{(n)}(\bfw)\in V_p$.

Suppose  that 
\begin{equation}
\label{eq:large M}
M > \varepsilon N\ge 2.
\end{equation}
Then take $r \le 2N/(M-1)$ as in Lemma~\ref{lem:Combin} and let $\cN$
be the set of $ n\in \{0,\ldots, N-1\}$ with
\begin{equation}
\label{eq:orb var}
\bfR_p^{(n)}(\bfw)\in V_p \mand  \bfR_p^{(n+r)}(\bfw) = 
\bfR_p^{(r)}\( \bfR_p^{(n)}(\bfw)\)\in V_p.
\end{equation}
By Lemma~\ref{lem:Combin},
\begin{equation}
\label{eq:large set N}
\# \cN \ge \frac{(M-1)^2}{4N} \gg \varepsilon^2 N.
\end{equation}

By~\eqref{eq:large M}, we have $r \ll  \varepsilon^{-1}$.

Since iterations of $\bfR$ 
generically escape $V$,  the set 
$ \{\bfz \in V \mid  \bfR^{(r)}(\bfz)\in V\}$ is finite.
This set is defined by the following $2s+1$ equations
\begin{equation}
\label{eq:Var}
\begin{split}
P_\nu(\bfX) = P_\nu\(\bfR^{(r)}(\bfX)\)=&0, \qquad \nu =1,   \ldots, s,\\
  1-X_0\prod_{i=1}^m G_{i,r}(\bfX) = &0
\end{split}
\end{equation}
where, as in the proof of  Theorem~\ref{thm:Cycles}, we write 
 $$
 R_i^{(r)}=\frac{F_{i,r}}{G_{i,r}},\qquad F_{i,r}, G_{i,r}\in\Z[\bfX],
 $$
with relative prime polynomials $F_{i,r}, G_{i,r}\in\Z[\bfX]$,
and introduce one more variable $X_0$. 

From now on, we denote by $c_{i}(\bfrho)$, $i=1,2, \ldots$, a sequence
of suitable constants depending only on the parameters in $\bfrho$.
By B{\'e}zout's theorem and the degree bound of
Lemmas~\ref{lem:HeighCompos} and~\ref{lem:HeighIter} we have
\begin{equation}
\label{eq:Bezout}
\begin{split}
  \# \{\bfz \in V & \mid  \bfR^{(r)}(\bfz)\in V\} \\&\le
D^s(Dd^{r}m^{r-1})^s ((d^{r}m^{r-1})^{m}+1)  \le
\exp\(\frac{c_{1}(\bfrho)}{\varepsilon}\).
\end{split}
\end{equation}
Using the height bound of Lemma~\ref{lem:HeighIter}, we also obtain
$$
\h(\bfR_i^{(r)}) \le  \exp\(\frac{c_{2}(\bfrho)}{\varepsilon}\), \qquad 
i = 1, \ldots, m.
$$
Therefore, by Lemma~\ref{lem:HeighCompos}, clearing the denominators,
we see that the $2s+1$ polynomials in~\eqref{eq:Var} have degree and
height of size bounded by $\exp({c_{3}(\bfrho)}{\varepsilon^{-1}})$.

Hence, by Theorem~\ref{thm:1}, there is a positive
integer $\fB$ 
with 
$$
\log \fB \le \exp\(\frac{c_{4}(\bfrho)}{\varepsilon}\)
$$
such that, if $p\nmid \fB$, then
$$
\# \{\bfz \in V_p \mid  \bfR_p^{(r)}(\bfz)\in  V_p\} =
\# \{\bfz \in V \mid  \bfR^{(r)}(\bfz)\in   V\}. 
$$
Since $N \le T(\bfw)$, the points $\bfR_p^{(n)}(\bfw)$, $n =
0,\ldots, N-1$, are pairwise distinct.  Hence, 
$$
  \# \cN\le \# \{\bfz \in V_p \mid  \bfR_p^{(r)}(\bfz)\in
  V_p\}. 
$$
From~\eqref{eq:orb var}, ~\eqref{eq:large set N} and~\eqref{eq:Bezout}
we deduce that
$$
\varepsilon^2 N \le \exp\(\frac{c_{1 }(\bfrho)}{\varepsilon}\).  
$$
Choosing 
$c( \bfrho) =  \max\{c_{4}(\bfrho), c_{1}(\bfrho)+1\} $,
this contradicts~\eqref{eq:eps rat}. Hence $M \le \varepsilon N$ and
the result follows.

\subsection{Proof of Corollary~\ref{cor:OrbInt}}
\label{sec:Orb vs Var} 

If $\fI_{\bfu,\bfv}(\bfR,\bfQ;p,N) $ is empty, the statement is
trivial. Otherwise, let $n_{0}\in \N$ be the smallest element in this
set. Then
$$
 \# \fI_{\bfu,\bfv}(\bfR,\bfQ;p,N) 
=  \# \fI_{\bfw,\bfw}(\bfR,\bfQ;p,N-n_{0}) 
$$
with $\bfw=\bfR^{(n_{0})}(\bfu)$. 
Moreover, 
$$
\fI_{\bfw,\bfw}(\bfR,\bfQ;p,N-n_{0}) 
= \fV_{\bfw}((\bfR(\bfX),\bfQ(\bfY)), V;p,N-n_{0}) 
$$
for the $2m$-dimensional system of rational functions
$$
(\bfR(\bfX),\bfQ(\bfY))
=(R_{1}(\bfX),\ldots ,R_{m}(\bfX),
 Q_{1}(\bfY)),\ldots,Q_{m}(\bfY))),
 $$
and the  variety $V$ defined by the polynomials
$$
P_j=X_j-Y_j,\qquad j=1,\ldots,m.
$$
The hypothesis the orbits of $\bfR$ and $\bfQ$ generically escape each other 
implies that the system $(\bfR(\bfX),\bfQ(\bfY))$ generically 
escapes the variety $V$. The statement then follows from
Theorem~\ref{thm:OrbIntVar}.

\subsection{Examples of iterations generically escaping  a
variety}
\label{sec:Ex Orbit Var}

Clearly, the problem of finding nontrivial pairs $(\bfR,V)$ consisting of a
system $\bfR$ of rational functions with iterations that generically escape
a variety $V$ is
interesting in its own. Here we give a family of examples of this
kind, as an application of a result of Dvir, Koll{\'a}r and
Lovett~\cite[Theorem~2.1]{DKL}.

Let $m =2 s$ be even and let 
$$
A=(a_{i,j})_{i,j}\in \Z^{s\times m}
$$ 
be an $s\times m$ matrix with 
integer entries such that any $s\times s$ minor is nonsingular. 
For instance,  one may  construct such a matrix  as a  Vandermonde or 
Cauchy matrix.

We now choose $2m$ positive integers with 
\begin{equation}
\label{eq:e ej}
 d_1 > \ldots > d_m \mand  e_1 > \ldots > e_m > d_1^s
\end{equation}
such that $ \gcd(d_ie_i, d_je_j) = 1$, $  1\le i,j \le m,  \ i \ne j$.

We consider the monomial system $\bfF = (X_1^{e_1}, \ldots,
X_m^{e_m})\in \Z[\bfX]^{m}$ and the variety $V\subset \C^{m}$ defined
by the $s$ polynomials
$$
P_j= \sum_{i=1}^{m} a_{j,i} X_i^{d_i}, 
\qquad j =1, \ldots, s.
$$
This variety is a complete intersection of degree
at most $d_1^s$. 

For any point $\bfw \in \C^m$ we have
$\(\bfw,\bfF^{(k)}(\bfw)\)\in V(\C)\times V(\C)$ if and only if $\bfw
\in U_k\cap V$, where $U_k$ is the variety defined by the polynomials
$$P_j(X_1^{e_1^k}, \ldots, X_m^{e_m^k}) =
\sum_{i=1}^{m} a_{j,i} X_i^{d_ie_i^k}, 
\qquad j =1, \ldots, s.
$$
As $V$ is of dimension $m-s = s$,
and recalling the conditions~\eqref{eq:e ej}, we see that
$$
d_1e_1^k > \ldots > d_me_m^k > d_1^s \ge \deg V.
$$
Therefore,~\cite[Theorem~2.1]{DKL} applies and yields the finiteness
of $U_k\cap V$. Hence, the iterations of the  monomial system $\bfF$ generically escape
the variety $V$, as desired.

\section{Orbits on Varieties 
under the Uniform Dynamical
Mordell-Lang Conjecture}
\label{sec:OrbVar UML}

\subsection{Varieties satisfying the uniform dynamical Mordell--Lang conjecture}
\label{sec:Orbit Var Alt}

Informally, the dynamical Mordell--Lang conjecture asserts that the
intersection of an orbit of an algebraic dynamical system (in affine
or projective space over a field of zero characteristic) with a given
variety is a union of a finite ``sporadic'' set and finitely many
arithmetic progressions. Among other sources, this conjecture stems
from the celebrated {\it Skolem-Mahler-Lech theorem}~\cite{BeLa}.

Here we consider a class of algebraic dynamical systems and varieties
that satisfy the following stronger uniform condition.

\begin{definition}
  Let  $\bfR\in \Q(\bfX)^{m}$ be a system of rational
  functions over $K$ and $V\subset \A_{\Q}^{m}$ an affine variety.  The
  intersection of the orbits $\bfR$ with $V$ is \emph{$L$-uniformly
    bounded} if there is a constant $L$ depending only on $\bfR$ and
  $V$ such that for all initial values $\bfw \in \ov \Q^m$,
$$
\# \left\{ n\in\N \mid \bfw_n \in V(\ov \Q)\right\}\le L,
$$
with  $ \bfw_n$ is as in~\eqref{eq:Orb R}.
\end{definition}

In this section, we reconsider the problem 
of~\S\ref{sec:OrbVar GenAvoid} of bounding the number of orbits of a given system of
rational functions lying in a variety satisfying this uniformity
condition.


The boundedness of the number of orbit elements that fall in a
variety, or more specialised questions of orbit intersections
(see~\S\ref{sec:Orb vs Var} where this link is made explicit), has
recently been an object of active study,
see~\cite{BeGiTu,BGKT1,BGKT2,GTZ1,GTZ2,OstSha,SiVi} and the references
therein.  Although we believe that the $L$-uniformly
boundedness condition is generically satisfied, proving it for general
classes of systems appear to be difficult.


\subsection{Systems of rational functions}

Here we add one parameter $L$ in the definition of  $\bfrho$, so instead 
of~\eqref{eq:rho} it is now given by 
$$
\bfrho = (d, D, h, H, L, m, s) .
$$
We also continue to use  $T(\bfw)$ as given by~\eqref{def:T(w)}. 
We obtain the following result
which is a version of Theorem~\ref{thm:OrbIntVar}.

\begin{theorem} 
\label{thm:OrbIntVar Alt}
Let $\bfR=(R_1,\ldots,R_m)$ be a system of $m\ge 2$ rational functions
in $\Q(\bfX)$ of degree at most $d\ge 2$ and of height at
most $h$.  Let $V$ be the affine algebraic variety defined by the
polynomials $P_1,\ldots,P_s\in\Z[\bfX]$ of degree at most
$D$ and height at most $H$. We also assume that the intersection of
orbits of $\bfR$ with $V$ is $L$-uniformly bounded. 
There
is a  constant $c(\bfrho) > 0$ such that, for any real $\varepsilon >0$, there exists $\fB \in \N$ with
$$
\log \fB   \le \exp\(\frac{c(\bfrho)}{\varepsilon}\)
$$
such that, if $p$ is a prime number not dividing $\fB$, then 
for any  integer 
$$
N\ge \frac{2L}{\varepsilon} +1
$$ 
and  any initial point $\bfw \in \ov \F_p^m$ 
with $T(\bfw) \ge N$, we have
$$
\frac{ \# \fV_{\bfw}(\bfR,V;p,N)  }{N}
 \le \varepsilon.
$$
\end{theorem}

\begin{rem}
  One can check that appropriate versions of Remarks~\ref{rem: eps and
    p} and~\ref{rem: slow} apply to Theorem~\ref{thm:OrbIntVar Alt} as
  well.
\end{rem}

%
%

\subsection{Proof of Theorem~\ref{thm:OrbIntVar Alt}}

We set 
$$
M = \fl{2 \varepsilon^{-1} L} +1,
$$
thus in particular $N \ge M$.

For each set $\cL \subseteq \{0, \ldots, M-1\}$ of cardinality 
$\# \cL = L+1$ we consider the system of equations  
$$
P_j(\bfR^{(k)})=P_j\(\frac{F_{1,k}}{G_{1,k}},\ldots,\frac{F_{m,k}}{G_{m,k}}\)=0, \qquad k \in \cL,\ j=1,\ldots,s.
$$
Let $X_0$ be an additional variable, set
$$
\Gamma_{0,k}=1 - X_0 \prod_{i=1}^m\prod_{j=1}^k G_{i,j} 
$$
and let $\Gamma_{j,k}$ be the numerator 
of $P_j(\bfR^{(k)})$.
We now study the following  system of equations in $m+1$ variables:
\begin{equation}
\label{eq:eqk}
\Gamma_{j,k}= 0, \qquad k \in \cL,\  j=0,\ldots,s.
\end{equation}

By Lemmas~\ref{lem:HeighCompos} and~\ref{lem:HeighIter}, we have
$$
\deg \Gamma_{0,k} \le 2d^km^k  \mand \deg \Gamma_{j,k}\le  d^kDm^k,
$$
which we combine in one bound
\begin{equation}
\label{eq:Gammadeg}
 \deg \Gamma_{j,k}\le  d^kDm^k+1,\qquad j=0,\ldots,s.
\end{equation}

Also, by Lemmas~\ref{lem:HeightFact} and~\ref{lem:HeighIter}, exactly as in the proof of Theorem~\ref{thm:Cycles}, we have
\begin{equation}
\label{eq:Gammaheight0}
\h(\Gamma_{0,k})  \ll_\bfrho (dm)^k.
\end{equation}
By Lemmas~\ref{lem:HeighCompos} and~\ref{lem:HeighIter} again, we also have
\begin{equation}
\label{eq:Gammaheight-jk}
\begin{split}
\h(\Gamma_{j,k}) &\le H+Dh\(1+d\frac{d^{k-1} m^{k-1}-1}{dm-1}\)\\ 
&\qquad\quad+dD(3dm+1)\frac{d^{k-1} m^{k-1}-1}{dm-1}\log(m+1)\\
&\qquad\qquad\qquad+D(3d^km^k+1)\log(m+1)\ll_\bfrho  (dm)^k.
\end{split}
\end{equation}
 For simplicity, we  use the bound~\eqref{eq:Gammaheight-jk} also for $\h(\Gamma_{0,k})$, even if we loose slightly in the final bound.
  
 By the assumption on $\bfR$ and $L$, the equations~\eqref{eq:eqk}
 have no common solution $\bfw \in \ov \Q^m$.  By 
 Theorem~\ref{thm:1} together with the bounds~\eqref{eq:Gammadeg},
 \eqref{eq:Gammaheight0} and~\eqref{eq:Gammaheight-jk} and the fact
 that $k \le M-1$, there exists $\fA_{\cL}\in \N$ with
$$
\log \fA_\cL  \ll_\bfrho (d^{M-1}Dm^{M-1}+1)^{3(m+1)+2}
$$
such that, if $p$ is a prime not dividing $\fA_{\cL}$, then the
reduction modulo~$p$ of the system of equations~\eqref{eq:eqk} has no
solution in $\ov \F_{p}^{m}$.

We now set 
$$
\fB = \prod_{\substack{\cL \subseteq \{0, \ldots, M-1\}\\
\# \cL = L+1}} \fA_\cL
$$
and note that $\fB \ge  1$ 
\begin{equation}
\label{eq:Alt Gen}
\log \fB \ll_\bfrho \binom{M}{L+1} (d^{M-1}Dm^{M-1}+1)^{3(m+1)+2} \le
\exp\(\frac{c_{1}(\bfrho)}{\varepsilon}\) 
\end{equation}
for a constant $c_{1}(\bfrho)$. 

Let $p$ be a prime with $p\nmid \fB$.  Suppose that for some $\bfu
\in \ov \F_p^m$ there are at least $\varepsilon N$ values of $n\in
\{0,\ldots, N-1\}$ with $ \bfR_p^{(n)}(\bfu)\in V_p$.  We recall
that $N \ge M$, so $\fl{N/M} +1 \le 2N/M$.  Therefore, there is a
nonnegative integer $i \le \fl{N/M}$ such that there are at least
$$
\frac{\varepsilon N}{\fl{N/M} +1} \ge \frac{1}{2} \varepsilon M  > L
$$ 
values of $n \in \{iM, \ldots, (i+1)M -1\}$ with $
\bfR_p^{(n)}(\bfu)\in V_p.  $ Take $L+1$ such values and write them
as
$$
s< s+ t_1 < \ldots <s+t_{L+1} < s+M
$$
where $s = iM$.
Then, for  $j=1,\ldots,s$ and $\nu=1, \ldots, L+1$, 
$$P_j\(\bfR_p^{(t_\nu)}\(\bfR_p^{(s)}(\bfu)\)\) = 0.
$$
So, setting $\bfw=\bfR_p^{(s)}(\bfu) \in \ov \F_p$, we obtain
$$
P_j\(\bfR_p^{(t_\nu)}\(\bfw\)\)
= 0
$$
for all such $j,\nu$. This implies that $p \mid \fA_\cL$ with $\cL =
\{t_1, \ldots, t_{L+1}\}$, and thus we obtain a contradiction.

%
%
%

\section{Some remarks} 
\label{sec:rem}

Clearly, our results depend on the growth of the degree and the height of the iterates~\eqref{eq:RatIter}.
When this growth is slower than ``generic'', one can expect stronger
bounds.  
For example, this is true for the following family of
systems which stems from that introduced in~\cite{OstShp1}, see
also~\cite{GomOstShp,OstShp2}. 

For $i=1,\ldots, m$, let
\begin{equation}
\label{eq:Polysyst}
F_i\in \Z[X_{i}, X_{i+1},\ldots,X_m]
\end{equation} 
be a ``triangular'' system of polynomials $F_i$ which do not depend on the first 
$i-1$ variables and 
with a term of the form $ g_iX_{i}X_{i+1}^{s_{i,i+1}} \ldots
X_m^{s_{i,m}}$ such that $g_{i}\in \Z\setminus \{0\}$, $\deg_{X_{i}}F_{i}=1$ and
$\deg_{X_{j}}F_{i}=s_{i,j}$, $j=i+1,\ldots, m$. 


Using the same idea as in~\cite{OstShp1}, similarly to the bounds
of \S\ref{sec:polyn-syst-height}, 
one can show that for the systems~\eqref{eq:Polysyst} 
the degree and the height of the  $k$th
iterate grow  polynomially and thus obtain stronger versions of our 
main results in \S\ref{sec:Period}, \S\ref{sec:OrbVar GenAvoid} and \S\ref{sec:OrbVar UML}.


Indeed, an inductive argument shows that for any  integer $k\ge 1$,
the polynomials $F_i^{(k)}$, $i =1,\ldots, m$, defined
by~\eqref{eq:RatIter}, are of degree and height at most
$$
d_{i,k} = O_{d,m}(k^{m-i}) \mand h_{i,k} = O_{d,h,m}(k^{m-i+2}). 
$$
In turn, one obtains a version of Theorem~\ref{thm:Cycles-Poly}  with an
  integer $\fA_k \ge 1$ satisfying 
$$
\log \fA_k \ll_{d,h,m} k^{m(3m-1)}
$$
and 
such that, if $p$ is a prime number not dividing $\fA_k$, 
then the reduction of $\bfF$ modulo $p$ has at most $O_{d,h,m}(k^{m(m-1)/2})$
periodic points of order~$k$.  Similarly, for systems the form~\eqref{eq:Polysyst}
a version of Theorem~\ref{thm:OrbIntVar} holds with 
$$
N \ge {c(\bfrho)}{\varepsilon^{-(m-1)s-2}}
\mand 
\log \fB \le c(\bfrho)    \varepsilon^{-m(3m-1)},
$$
while a version of Theorem~\ref{thm:OrbIntVar Alt} holds with
$$
\log \fB   \le c(\bfrho)  \varepsilon^{-(m-1)s(L+1) + m + L +1}
$$
and the same value of $N$. 

The polynomial systems of the form~\eqref{eq:Polysyst} have been
generalised in various directions, including their rational function
analogues \cite{GomOstShp, HassPropp}. It is expected that similar
improvements hold for all these systems as well.

\end{document}